\documentclass[leqno,11pt]{amsart}
\usepackage{amsfonts}
\usepackage{amsmath}
\usepackage{amsfonts}
\usepackage{amssymb}
\usepackage{amscd}
\newtheorem{theorem}{Theorem}

\newtheorem{corollary}[theorem]{Corollary}

\newtheorem{definition}[theorem]{Definition}
\newtheorem{example}[theorem]{Example}
\newtheorem*{example*}{Example}

\newtheorem{lemma}[theorem]{Lemma}

\newtheorem{proposition}[theorem]{Proposition}
\newtheorem{remark}[theorem]{Remark}
\newtheorem{remarks}[theorem]{Remarks}

\numberwithin{theorem}{section}

\numberwithin{equation}{section}

\begin{document}

\title{Ricci almost solitons}
\author[S.Pigola]{Stefano Pigola}
\address{Dipartimento di Fisica e Matematica\\
Universit\`a dell'Insubria - Como\\
via Valleggio 11\\
I-22100 Como, ITALY}
\email{stefano.pigola@uninsubria.it}

\author[M. Rigoli]{Marco Rigoli}
\address{Dipartimento di Matematica\\
Universit\`a degli Studi di Milano\\
via Saldini 50\\
I-20133 Milano, ITALY}
\email{marco.rigoli@unimi.it}

\author[M. Rimoldi]{Michele Rimoldi}
\address{Dipartimento di Matematica\\
Universit\`a degli Studi di Milano\\
via Saldini 50\\
I-20133 Milano, ITALY}
\email{michele.rimoldi@unimi.it}

\author[A.G. Setti] {Alberto G. Setti}
\address{Dipartimento di Fisica e Matematica
\\
Universit\`a dell'Insubria - Como\\
via Valleggio 11\\
I-22100 Como, ITALY}
\email{alberto.setti@uninsubria.it}



\date{\today}
\maketitle

\begin{abstract}
We introduce a natural extension of the concept of gradient Ricci soliton: the Ricci almost soliton. We provide existence and rigidity results, we deduce a-priori curvature estimates and isolation phenomena, and we investigate some topological properties. A number of differential identities involving the relevant geometric quantities are derived. Some basic tools from the weighted manifold theory such as general weighted volume comparisons and maximum principles at infinity for diffusion operators are discussed.
\end{abstract}

\tableofcontents

\section*{Introduction}
Let $\left(M,\left\langle \,,\right\rangle\right)$ be a Riemannian manifold. A Ricci soliton structure
on $M$ is the choice of a smooth vector field $X$ (if any) satisfying the soliton equation
\begin{equation}\label{0.1}
Ric +\frac{1}{2}L_{X}\left\langle \, ,\right\rangle=\lambda\left\langle \,,\right\rangle
\end{equation}
for some constant $\lambda \in \mathbb{R}$. Here, $Ric$ denotes
the Ricci tensor of $M$ and $L_{X}$ stands for the Lie derivative
in the direction of $X$.

In the special case where $X=\nabla f$ for some smooth function $f:M\rightarrow\mathbb{R}$,
we say that $\left(M,\left\langle \,,\right\rangle,\nabla f\right)$ is a gradient Ricci soliton
with potential $f$. In this case the soliton equation (\ref{0.1}) reads
\begin{equation}\label{0.2}
Ric + \rm{Hess}\left(f\right)=\lambda\left\langle \,,\right\rangle.
\end{equation}
Clearly, equations (\ref{0.1}) and (\ref{0.2}) can be considered as perturbations of the Einstein equation
\begin{equation*}
Ric=\lambda\left\langle\, ,\right\rangle
\end{equation*}
and reduce to this latter in case $X$ or $\nabla f$ are Killing vector fields.
When $X=0$ or $f$ is constant we call the underlying Einstein manifold a trivial
Ricci soliton.

As in the case of Einstein manifolds, Ricci solitons exhibit a
certain rigidity. This is expressed by  triviality and classification
results, or by curvature estimates.

For instance, in the compact case it is well known that an expanding (or steady) Ricci soliton is 
necessarily trivial (see, e.g., \cite{ELNM}). Generalizations to the
complete, non-compact setting can be found in the very recent \cite{PRiS}.

On the other hand, since the appearance of the seminal works by R.
Hamilton, \cite{H}, and  G. Perelman, \cite{Pe}, the classification of shrinking gradient Ricci solitons
has become the subject of a rapidly increasing investigation. In this direction, we limit ourselves to quote
the far-reaching \cite{CCZ} by H.-D. Cao, B.-L. Chen and X.-P. Zhu where a complete classification in the three-dimensional case
is given, \cite{Zh-conf} by Z.-H. Zhang for the extension in the conformally flat, higher dimensional case, and the very recent
\cite{MoSe} by O. Monteanu and N. Sesum where, on the base of rigidity works by P. Petersen and W. Wylie, \cite{PW1}, \cite{PW2}, and
M. Fern\'andez-L\'opez and E.Garc\'{\i}a-R\'{\i}o, \cite{FLGR-rigidity}, the authors extend Zhang classification result to
complete shrinkers with harmonic Weyl tensor.
The classification of expanding Ricci solitons appears to be more difficult and relatively few results are known. For instance,
the reader may consult \cite{PW1} for the case of constant scalar curvature expanders.

As an instance of curvature estimates, we quote the recent papers
by B.-L. Chen, \cite{Ch}, and by  Z.-H. Zhang, \cite{Zh-complete},
where it is shown that the scalar curvature of
any gradient Ricci soliton is bounded below. In another direction,
upper and lower  estimates for the infimum of the scalar curvature
of a gradient Ricci soliton are obtained in  \cite{PRiS}, where some
triviality and  rigidity  results at the endpoints are also discussed.

\medskip

In this paper we propose an extension of the concept of  Ricci soliton that, as we are going to explain,
appears to be natural and meaningful. First of all we set the following
\begin{definition}
We say that $\left(M,\left\langle\, ,\right\rangle\, ,\nabla f\right)$ is a gradient Ricci almost
soliton (almost soliton for short) with potential $f$ and soliton function $\lambda$ if (\ref{0.2}) holds on $M$ with
$\lambda$ a smooth function on $M$.
\end{definition}

Clearly, the above definition generalizes the notion of gradient Ricci solitons. One could
consider almost Ricci solitons which are not necessarily gradient, replacing the Hessian
of $f$ with the  Lie derivative $\frac 12 L_X \langle\, ,\rangle$ of the metric  along a vector field, and
study the properties of this new object. For instance, it is an interesting problem to find
under which conditions an almost Ricci soliton is necessarily gradient.  This is going to be
the subject of a forthcoming paper. Here we are going to deal with gradient almost Ricci solitons.

We also note that that generalizations in different directions have been recently considered.  For instance,
J. Case, Y.-J. Shu and G. Wei  introduce in \cite{CWS} the concept of a ``quasi Einstein manifold'', i.e., a Riemannian manifold whose modified Bakry-Emery Ricci tensor is constant. This definition originates from the study of usual Einstein manifolds that are realized as warped products. Further generalizations have been considered by G. Maschler in \cite{Ma}, where equation (1) is replaced by what the author calls the``Ricci-Hessian equation", namely,
\[
\alpha \mathrm{Hess}\, f + \mathrm{Ric} = \gamma \langle \,,\rangle,
\]
where $\alpha$ and $\gamma$ are functions. Note that since the author is interested in conformal changes
of K\"ahler-Ricci solitons which give rise to new K\"ahler metrics, the presence of the function $\alpha$ is vital in his investigation.
\par

Extending to our new setting the soliton terminology, we say that the gradient Ricci
almost soliton $\left(M,\left\langle\, ,\right\rangle\, ,\nabla f\right)$ is shrinking,
steady or expanding if respectively $\lambda$ is positive, null or negative on $M$.
If $\lambda$ has no definitive sign the gradient Ricci almost soliton will be called
indefinite. In case $f$ is constant the almost soliton is called trivial and if $\dim M\geq 3$
the underlying manifold $\left(M,\left\langle\,,\right\rangle\right)$ is
Einstein by Schur Theorem. This also suggests that for an almost soliton an appropriate
terminology could be that of an almost Einstein manifold.

In view of the fact that the soliton function $\lambda$ is not necessarily constant, one expects that a certain
flexibility on the almost soliton structure is allowed and, consequently, the existence of almost
solitons is easier to prove than in the classical situation. This feeling is confirmed
in Section \ref{Ex_Einst} below where we shall give a number of different examples of almost solitons, showing in
particular that all the previous possibilities (shrinking and expanding) with a non-constant
soliton function $\lambda$ can indeed occur.
On the other hand, the rigidity result contained in Theorem \ref{th_classification} below indicates that almost solitons
should reveal a reasonably broad generalization of the fruitful concept of classical soliton.
In particular, one obtains that not every complete manifold supports an almost soliton structure; see Example \ref{not-almostsolitons}.

\medskip

The investigation in this paper is mainly concerned with triviality and pointwise curvature estimates
of gradient Ricci almost solitons in case $\left(M,\left\langle ,\right\rangle\right)$ is a complete, connected manifold.
From now on we let $m=\dim M$; we fix an origin $o\in M$ and we let $r\left(x\right)$ denote the
distance function from $o$. $B_{r}$ and $\partial B_{r}$ are respectively the geodesic ball of
radius $r$ centered at $o$ and its boundary. Given the potential function
$f\in C^{\infty}\left(M\right)$ we consider the weighted manifold
$\left(M,\left\langle\, ,\right\rangle,e^{-f}d\rm{vol}\right)$, where $d\rm{vol}$ is the
Riemannian volume element. We set
\[
\mathrm{vol}_{f}\left(  B_{r}\left(  p\right)  \right)  =\int_{B_{r}\left(
p\right)  }e^{-f}d\mathrm{vol},\qquad\mathrm{vol}_{f}\left(  \partial
B_{r}\left(  p\right)  \right)  =\int_{\partial B_{r}\left(  p\right)  }%
e^{-f}d\mathrm{vol}_{m-1},
\]
where $d\mathrm{vol}_{m-1}$ stands for the $\left(  m-1\right)  $-Hausdorff
measure.
Finally we call $f$-laplacian,
$\Delta_{f}$, the diffusion operator defined on $u$ by
\begin{equation*}
\Delta_{f}u=e^{f}div\left(e^{-f}\nabla u\right)=\Delta u-\left\langle \nabla f,\nabla u\right\rangle
\end{equation*}
which is clearly symmetric on $L^2\left(M,e^{-f}d\rm{vol}\right)$.

Note that in the terminology of weighted manifolds the LHS of (\ref{0.2}) is the Bakry-Emery
Ricci tensor of $\left(M, \left\langle \,,\right\rangle, e^{-f}d\rm{vol}\right)$ that is usually
indicated with $Ric_{f}$.
\\
We are now ready to state our first result.
\begin{theorem}\label{A}
Let $\left(M,\left\langle \,,\right\rangle,\nabla f\right)$ be a complete, expanding gradient
Ricci almost soliton with soliton function $\lambda$. Let $\alpha$, $\sigma$, $\mu\in \mathbb R$
be such that
\begin{equation*}
\alpha>-2\;;\;0\leq\sigma\leq 2/3
\end{equation*}
\begin{equation}\label{0.6}
\min\left\{0,-\alpha\right\}\leq \mu\leq \left\{
\begin{array}{lll}
1-3\sigma/2&\textrm{if}&\sigma\geq\alpha\\
1-\sigma-\alpha/2&\textrm{if}&\sigma<\alpha\\
\end{array}
\right.
\end{equation}
Assume
\begin{equation}\label{0.7}
\limsup_{r\left(x\right)\rightarrow +\infty}\frac{\left|\nabla f\right|^2}{r\left(x\right)^\sigma}\left\{
\begin{array}{lll}
=0&\textrm{if}&0<\sigma\leq 2/3\\
<+\infty&\textrm{if}&\sigma=0\\
\end{array}
\right.
\end{equation}
\begin{equation}\label{0.8}
-\left(m-1\right)B^{2}\left(1+r\left(x\right)^{2}\right)^{\frac{\alpha}{2}}\leq \lambda\left(x\right)\leq -\left(m-1\right)A^{2}\left(1+r\left(x\right)^{2}\right)^{-\frac{\mu}{2}}
\end{equation}
on $M$ for some constants $B\geq A>0$.\\
Suppose either $m=2$ or
\begin{equation}\label{0.9}
\left\langle \nabla f,\nabla \lambda\right\rangle\leq 0 \;\text{on} \;M.
\end{equation}
Then, the almost soliton is trivial.
\end{theorem}

Note that (\ref{0.6}) implies that (\ref{0.8}) is meaningful.
\begin{corollary}\label{cor_A}
Let $\left(M,\left\langle \,,\right\rangle,\nabla f\right)$ be a complete, expanding gradient
Ricci soliton such that
\begin{equation}\label{0.10}
\limsup_{r\left(x\right)\rightarrow +\infty}\frac{\left|\nabla f\right|^{2}}{r\left(x\right)^{\sigma}}\left\{
\begin{array}{ll}
=0&0<\sigma\leq\frac{2}{3}\\
<+\infty&\sigma=0.\\
\end{array}
\right.
\end{equation}
Then the soliton is trivial.
\end{corollary}

The case $\sigma=0$ of Corollary \ref{cor_A} has been proved in \cite{PRiS}.
The next result extends Theorem 3 in \cite{PRiS} to the case of almost solitons; see also \cite{SZhang}.
Note that, contrary to \cite{PW2}, we do not assume that the scalar curvature is either
constant nor bounded.

\begin{theorem}\label{B}
Let $\left(M,\left\langle\, ,\right\rangle, \nabla f\right)$ be a complete gradient Ricci
almost soliton with scalar curvature S and soliton function $\lambda$ such that
$\Delta \lambda \leq 0$ on $M$. Set
\begin{equation*}
S_{*}=\inf_{M}S,\,\,\,\,\,
\lambda_{*}=\inf_{M}\lambda, \,\, \,\,\, \,
\lambda^{*}=\sup_{M}\lambda.
\end{equation*}
\begin{enumerate}
    \item[(i)] If the almost soliton is expanding with
$\lambda_* \leq \lambda \leq 0,$ $\lambda\not\equiv 0$, then $m\lambda_* \leq S_* < 0$.
Moreover, if $m\geq 3$ and there exists $x_o$ such that $S(x_o)= S_*=m\lambda_* $,
then the soliton is trivial and $M$ is Einstein.

\item[(ii)] If the almost soliton is a steady soliton then $S_* = 0.$
Morever, if $m\geq 3$ and there exists $x_o$ such that $S(x_o)=0,$
then $M$ is a cylinder over a totally geodesic hypersurface.

\item[(iii)]
If the almost soliton is shrinking with $0\leq \lambda \leq
\lambda^*,$ $\lambda\not\equiv 0$,
then $0\leq S_*\leq m\lambda^*.$ Moreover if $m\geq 3$ and there exists $x_o$ such that
$S(x_o)=S_*=0$ then $M$ is isometric to the standard Euclidean space. Finally if $S_* =
m\lambda^*$ and $(M, \langle\, , \rangle, e^{-f}d\mathrm{vol})$ is $f$-parabolic, then the almost soliton
is trivial and $\left(M, \left\langle\, ,\right\rangle\right)$ is compact
Einstein. This latter case occurs in particular if
\begin{equation*}
A^2\left(1+r\left(x\right)\right)^{-\mu}\leq \lambda\left(x\right)\leq\lambda^{*}<+\infty
\end{equation*}
on $M$ for some $A>0$, $0\leq\mu<1$.
\end{enumerate}
\end{theorem}
Note that the case $\mu=0$ contains, of course, the soliton case.
\begin{corollary}\label{cor_B}
In the assumptions of Theorem \ref{B}, in cases (i) and (iii),
$\left(M,\left\langle\,  ,\right\rangle\right)$ has non negative scalar curvature.
\end{corollary}
In the next result we shall assume the validity of a weighted Poincar\'e-Sobolev inequality on $M$.
\begin{theorem}\label{C}
Let $\left(M,\left\langle\, , \right\rangle,\nabla f\right)$ be a complete, indefinite,
gradient Ricci almost soliton with soliton function $\lambda$. For some $0\leq\alpha<1$
assume on $M$ the validity of
\begin{equation}\label{0.15}
\int_{M}\left|\nabla \varphi\right|^{2}e^{-f}\geq S\left(\alpha\right)^{-1}\left\{\int_{M}\left|\varphi\right|^{\frac{2}{1-\alpha}}e^{-f}\right\}^{1-\alpha}
\end{equation}
for all $\varphi\in C^{\infty}_{0}\left(M\right)$ and some constant $S\left(\alpha\right)>0$.
Suppose that
\begin{equation}\label{0.16}
\int_{B_{r}}\left|\nabla f\right|^{p}e^{-f}=o\left(r^{2}\right)
\end{equation}
as $r\rightarrow+\infty$, for some $p>1$, and that
\begin{equation*}
\left\|\lambda_{+}\left(x\right)\right\|_{L^{\frac{1}{\alpha}}\left(M,e^{-f}d\rm{vol}\right)}
<\frac{4}{S\left(\alpha\right)}\frac{p-1}{p^{2}},
\end{equation*}
with $\lambda_{+}\left(x\right)=\max\left\{0,\lambda\left(x\right)\right\}$.
Suppose that either $m=2$ or (\ref{0.9}) is satisfied. Then the almost soliton is trivial and, when $m\geq 3$,
$M$ is Einstein with non-positive Ricci curvature.
\end{theorem}
As an immediate consequence we obtain,
\begin{corollary}\label{cor_C}
Let $\left(M,\left\langle ,\right\rangle\nabla f\right)$ be a complete, expanding,
gradient Ricci soliton and assume that
(\ref{0.15}) and (\ref{0.16}) hold for some $0\leq\alpha<1$ and  $p>1$. Then the soliton is trivial.
\end{corollary}

\begin{remarks} \rm{
  (a) Note that, according to the variational characterization of the bottom of the spectrum of the $f$-Laplacian,  assumption (\ref{0.15}) with $\alpha=0$ means
\[
\ \lambda^{-\Delta_{f}}_{1}(M,e^{-f}d\rm{vol})>0.
\]
Thus, in particular, inequality (\ref{0.15}) with $\alpha=0$ holds if the almost soliton $\left(M,\left\langle ,\right\rangle,\nabla f\right)$ is expanding and satisfies:

$$Sec_{rad}\leq-K\leq0 \text{  and  } \frac{\partial f}{\partial r}\leq 0.$$
This follows from Theorem 3.4 in \cite{S}.

\noindent (b) Condition (\ref{0.15}) implies that $\mathrm{vol}_{f}\left(M\right)=+\infty$.
Indeed, let $\varphi\in C^{\infty}_{c}$ be such that $\varphi=1$ on $B_{R}$,
$\varphi=0$ off $B_{2R}$, $\left|\nabla \varphi\right|\leq\frac{C}{R}$. Then, from (\ref{0.15})
\begin{align*}
\frac{C^2}{R^2}\int_{B_{2R}\setminus B_{R}}e^{-f}\geq&\int_{M}\left|\nabla\varphi\right|^{2}e^{-f}\\
\geq&S\left(\alpha\right)^{-1}\left(\int_{M}\left|\varphi\right|^{\frac{2}{1-\alpha}e^{-f}}\right)^{1-\alpha}\\ \geq& S\left(\alpha\right)^{-1}\left(\int_{B_{R}}e^{-f}\right)^{1-\alpha},
\end{align*}
i.e
\[
\ \frac{C^{2}}{R^{2}}\left\{vol_{f}\left(B_{2R}\right)-vol_{f}\left(B_{R}\right)\right\}\geq S\left(\alpha\right)^{-1}vol_{f}\left(B_{R}\right)^{1-\alpha}.
\]
}
\end{remarks}
Denoting by  $T=Ric_{M}-\frac{1}{m}S\left\langle ,\right\rangle$ the trace free Ricci tensor of $\left(M,\left\langle ,\right\rangle\right)$ the next result is a gap theorem for the values of
\[
\ \left|T\right|^{*}=\sup_{M}\left|T\right|.
\]
\begin{theorem}\label{D}
Let $\left(M,\left\langle ,\right\rangle,\nabla f\right)$ be a
complete, conformally flat almost soliton with
scalar curvature $S$, trace free Ricci tensor $T$ and soliton
function $\lambda$ such that
\begin{equation}\label{0.18}
\left\langle Hess\left(\lambda\right),T\right\rangle\geq 0
\end{equation}
on $M$. Assume $m=\dim M\geq 3$,
\begin{equation}\label{0.19}
S^{*}=\sup_{M}S<+\infty,
\end{equation}
\begin{equation}\label{0.21}
\lambda_{*}=\inf_{M}\lambda>-\infty.
\end{equation}
Then either $\left(M,\left\langle ,\right\rangle\right)$ is Einstein and the classification of Theorem \ref{th_classification} below applies or
\begin{equation*}
\left|T\right|^{*}\geq \frac{1}{2}\left(\sqrt{m\left(m-1\right)}\lambda_{*}-S^{*}\frac{m-2}{\sqrt{m\left(m-1\right)}}\right).
\end{equation*}
\end{theorem}
Our results will follow from considering elliptic equations or inequalities for various geometric quantities on almost solitons and rely on analytic techniques. This is the same philosophy used, for instance in \cite{ELNM}, \cite{PW1}. More specifically, we will see that the differential (in)equalities at hand naturally involve the $f$-Laplace operator. Since the almost soliton equation means precisely that the $f$-Bakry-Emery Ricci curvature is proportional to the metric tensor, we are naturally led to introduce a number of weighted manifolds tools whose range of applications go beyond the investigation of almost solitons. This point of view is in the spirit of \cite{WW}. An important instance of these tools is represented by the maximum principles at infinity under both weighted Ricci lower bounds and  weighted volume growth conditions. By way of example, we shall observe the validity of the following

\begin{theorem}\label{OY}
Let $(M,\left\langle ,\right\rangle,e^{-f}d\rm{vol})$ be a complete weighted manifolds whose Bakry-Emery Ricci tensor satisfies

$$Ric_f \geq -(m-1)G(r(x))$$

where $G$ is a smooth function on $\left[0,+\infty\right)$ satisfying
\begin{equation*}
\begin{array}{lll}
&\left(i\right)\,G\left(0\right)>0&\left(ii\right)\,G^{\prime}\left(t\right)\geq 0 \textrm{\,\,on\,\,} \left[0,+\infty\right)\\
&\left(iii\right)G\left(t\right)^{-\frac{1}{2}}\notin L^{1}\left(+\infty\right)&\left(iv\right)\, \limsup_{t\rightarrow+\infty}\frac{tG\left(t^{\frac{1}{2}}\right)}{G\left(t\right)}<+\infty.
\end{array}
\end{equation*}
Assume also that
\begin{equation*}
\left|\nabla f\right|\leq CG\left(r\right)^{1/2}.
\end{equation*}
Then, for every smooth function $u$ such that $\sup_M u = u^{*} <+\infty$ there exists a sequence $\{x_n\}$ along which
\begin{equation*}
\begin{array}{llll}
&\left(i\right)\,u\left(x_{k}\right)>u^{*}-\frac{1}{k};&\left(ii\right)\,\left|\nabla u\left(x_{k}\right)\right|<\frac{1}{k};&\left(iii\right)\,\Delta_{f}u\left(x_{k}\right)<\frac{1}{k}.
\end{array}
\end{equation*}

\end{theorem}

The paper is organized as follows. In section 1 we provide some examples of almost gradient solitons and we prove the rigidity result contained in Theorem \ref{th_classification}. Section 2 is devoted to derive the basic elliptic equations that we shall use in the proofs of our results. Some of these equations, for $\lambda=const$, are well known, see for instance \cite{ELNM}, \cite{PW1}. Others seem to be new, even in this case (see for instance (\ref{2.25}) below). In section 3 we give some improved versions of Laplacian and volume comparison theorems obtained by Wei and Wylie, \cite{WW}, and we recall a version of a weak maximum principle, recently proved in \cite{MRS}, appropriate for our present purposes. We also prove a weighted version of a result in \cite{PRS-Memoirs} from which Theorem \ref{OY} above immediately follows. The proofs of the main geometric results are contained in sections 3, 4, 5, 6. We end with section 7 where we extend to almost solitons some topological results known in the classical case.

The authors are grateful to Manuel Fern\'andez-L\'opez for having sent them the preprints \cite{LR-Preprint} and \cite{FLGR-rigidity}.

\section{Examples and rigidity of Einstein almost solitons}\label{Ex_Einst}

Let $M=I\times_{g}\Sigma$ \ denote the $g$-warped product of the real interval
$I\subseteq\mathbb{R}$ with $0\in I,$ and the Riemannian manifold $\left(
\Sigma,\left( \, ,\right)  _{\Sigma}\right)  $ of dimension $\dim\Sigma=m$.
Namely, the $\left(  m+1\right)  $-dimensional, smooth product manifold
$I\times\Sigma$ is endowed with the metric%
\begin{equation*}
\left\langle \,,\right\rangle =dt\otimes dt+g\left(  t\right)
^{2}\left(\,,\right)  _{\Sigma},
\end{equation*}
where $t$ is a global parameter of $I$ and $g:I\rightarrow\mathbb{R}^+_0$
is a smooth function. Using the moving-frame formalism, the geometry of $M$
can be described as follows.

Fix the index convention $1\leq i,j,k,l,t...\leq m$ and $1\leq\alpha
,\beta,\gamma,...\leq m+1$. Let $\left\{  e_{j}\right\}  $ be a local
orthonormal frame of $\Sigma$ with dual frame $\left\{  \theta^{j}\right\}  $
so that $\left(  \,,\right)  _{\Sigma}=\sum\theta^{j}\otimes\theta^{j}.$ We
denote the corresponding connection $1$-forms by $\theta_{k}^{j}=-\theta
_{j}^{k}$ and the curvature $2$-forms by $\Theta_{j}^{i}=-\Theta_{i}^{j}$.
Accordingly, the structural equations of $\Sigma$ are
\begin{align*}
d\theta^{j} &  =-\theta_{k}^{j}\wedge\theta^{k}
\\
d\theta_{i}^{j} &  =-\theta_{k}^{j}\wedge\theta_{i}^{k}+\Theta_{i}%
^{j}.\nonumber
\end{align*}
Furthermore, the curvature forms are related to the (components of) the
Riemann tensor by%
\begin{equation*}
\Theta_{i}^{j}=\frac{1}{2}\text{ }^{\Sigma}R_{ikl}^{j}\theta^{k}\wedge
\theta^{l}.
\end{equation*}
Let us introduce the local orthonormal coframe $\left\{  \varphi^{\alpha
}\right\}  $ on $M$ such that%
\begin{equation*}
\varphi^{j}=g\left(  t\right)  \theta^{j}\text{, }\varphi^{m+1}%
=dt.
\end{equation*}
The corresponding connection and curvature forms are denoted, respectively, by
$\varphi_{\beta}^{\alpha}=-\varphi_{\alpha}^{\beta}$ and $\Phi_{\beta}%
^{\alpha}=-\Phi_{\alpha}^{\beta}=\frac{1}{2}$ $^{M}R_{\beta\delta\gamma
}^{\alpha}\varphi^{\delta}\wedge\varphi^{\gamma}$. A repeated use of exterior
differentiations of $\varphi^{\alpha}$ and $\varphi_{\beta}^{\alpha}$ and of
the structure equations of $M$ and $\Sigma$, together with the well known characterization of the Levi Civita connection forms, yield%
\begin{align}
\varphi_{j}^{k} &  =\theta_{j}^{k}\label{warp6}\\
\varphi_{m+1}^{k} &  =\frac{g^{\prime}}{g}\varphi^{k}=-\varphi^{m+1}_{k},\nonumber
\end{align}
and consequently,
\begin{align}
\Phi_{j}^{k} &  =-\left(  \frac{g^{\prime}}{g}\right)  ^{2}\varphi^{k}%
\wedge\varphi^{j}+\Theta_{j}^{k}\nonumber\\
\Phi_{k}^{m+1} &  =\left\{  \left(  \frac{g^{\prime}}{g}\right)  ^{2}+\left(
\frac{g^{\prime}}{g}\right)  ^{\prime}\right\}  \varphi^{k}\wedge\varphi
^{m+1}=\frac{g^{\prime\prime}}{g}\varphi^{k}\wedge\varphi^{m+1}=-\Phi^{k}_{m-1}.\nonumber
\end{align}
Let $\left\{  E_{\alpha}\right\}  $ denote the dual frame of $\left\{
\varphi^{\alpha}\right\}  $ so that $E_{j}=g\left(  t\right)  ^{-1}e_{j}.$
Then,%
\[
^{M}Ric_{\alpha\beta}=\Phi_{\alpha}^{\gamma}\left(  E_{\gamma},E_{\beta
}\right)  \text{, and }^{\Sigma}Ric_{kt}=g^{2}\Theta_{k}^{j}\left(
E_{j},E_{t}\right)  .
\]
It follows from (\ref{warp6})  that%
\begin{align}
^{M}Ric_{kt} &  =\left\{  -\left(  m-1\right)  \left(  \frac{g^{\prime}}%
{g}\right)  ^{2}-\frac{g^{\prime\prime}}{g}\right\}  \delta_{kt}+\frac
{1}{g^{2}}\text{ }^{\Sigma}Ric_{kt}\label{warp7}\\
^{M}Ric_{m+1t} &  =0\nonumber\\
^{M}Ric_{m+1\text{ }m+1} &  =-m\frac{g^{\prime\prime}}{g}.\nonumber
\end{align}
In light of these relations, we have that $M$ is Einstein with $^{M}%
Ric=-mc\left\langle ,\right\rangle $ if and only if%
\begin{equation*}
^{\Sigma}Ric_{kt}=\left\{  \left(  m-1\right)  \left(  \frac{g^{\prime}}%
{g}\right)  ^{2}+\frac{g^{\prime\prime}}{g}-mc\right\}  g^{2}\delta
_{kt},
\end{equation*}
and%
\begin{equation}
g^{\prime\prime}=cg.\label{warp7b}%
\end{equation}
Therefore%
\begin{equation}
^{\Sigma}Ric_{kt}=-\left(  m-1\right)  \left(  -g^{\prime}{}^{2}%
+cg^{2}\right)  \delta_{kt}.\label{warp7c}%
\end{equation}
We explicitly note that the general solution of (\ref{warp7b}) is given by%
\begin{equation}
g\left(  t\right)  =g^{\prime}\left(  0\right)  \mathrm{sn}_{-c}\left(
t\right)  +g\left(  0\right)  \mathrm{cn}_{-c}\left(  t\right)
,\label{warp7d}%
\end{equation}
where%
\begin{equation*}
\mathrm{sn}_{k}\left(  t\right)  =\left\{
\begin{array}
[c]{lll}%
\frac{1}{\sqrt{-k}}\sinh\left(  \sqrt{-k}t\right)&\textrm{if}& k<0\\
t&\textrm{if}& k=0\\
\frac{1}{\sqrt{k}}\sin\left(  \sqrt{k}t\right)&\textrm{if}& k>0.
\end{array}
\right.
\end{equation*}
and%
\begin{equation*}
\mathrm{cn}_{k}\left(  t\right)  ={\mathrm{sn}}^{\prime}_{k}\left(
t\right)  .
\end{equation*}
Inserting (\ref{warp7d}) into (\ref{warp7c})  we obtain the following

\begin{lemma}
\label{lemma_es}
Let $\left(  \Sigma,\left( \, ,\right)  _{\Sigma}\right)  $ be a Riemannian
manifold of dimension $m$. Consider the warped product $M=I\times_{g}\Sigma$
where $0\in I\subseteq\mathbb{R}$ and $g:I\rightarrow\mathbb{R}^{+}$ is
a smooth function. Then, $M$ is Einstein with%
\[
^{M}Ric=-m  c\left\langle\,  ,\right\rangle,\,\,\,\,c\in\mathbb{R},%
\]
if and only
\begin{equation}
g\left(  t\right)  =g^{\prime}\left(  0\right)  \mathrm{sn}_{-c}\left(
t\right)  +g\left(  0\right)  \mathrm{cn}_{-c}\left(  t\right)
\label{warp7l}%
\end{equation}
and $\Sigma$ is Einstein with%
\begin{equation}
^{\Sigma}Ric=-\left(  m-1\right)  \left\{  -g^{\prime}\left(  0\right)
^{2}+cg\left(  0\right)  ^{2}\right\}  \left( \, ,\right)  _{\Sigma
}.\label{warp7i}%
\end{equation}
\end{lemma}

Now, consider a smooth function $f:M\rightarrow\mathbb{R}$ of the form
$f\left(  t,x\right)  =f\left(  t\right)  $. Its Hessian expresses as%
\begin{equation}
\mathrm{Hess}\left(  f\right)  =f^{\prime}\frac{g^{\prime}}{g}\sum\varphi
^{k}\otimes\varphi^{k}+f^{\prime\prime}\varphi^{m+1}\otimes\varphi^{m+1}.
\label{warp8}%
\end{equation}
Thus, in case $M$ is an Einstein manifold with $^{M}Ric=-mc\left\langle\, ,\right\rangle$
(hence $\Sigma$ is so), the almost Ricci
soliton equation on $M$ with respect to the potential $f\left(  x,t\right)
=f\left(  t\right)  $ reads%
\begin{equation*}
\left\{
\begin{array}
[c]{l}%
f^{\prime}\frac{g^{\prime}}{g}-mc=\lambda\\
f^{\prime\prime}-mc=\lambda.
\end{array}
\right.
\end{equation*}
Integrating this latter we deduce%
\begin{equation}
\left\{
\begin{array}
[c]{l}%
f\left(  t\right)  =a\int_{0}^{t}g\left(  s\right)  ds+b\\
\lambda\left(  t\right)  =ag^{\prime}\left(  t\right)  -mc,
\end{array}
\right.  \label{warp10}%
\end{equation}
for some constants \thinspace$a,b\in\mathbb{R}$. Summarizing, we have obtained
the following examples of Einstein, almost Ricci solitons.

\begin{proposition}\label{es2}
Let $g\left(  t\right)  :I\rightarrow\mathbb{R}^{+}$ be the  smooth
function defined  in (\ref{warp7l}), $0\in I\subseteq\mathbb{R}$. Let $\left(
\Sigma,\left( \, ,\right)  _{\Sigma}\right)  $ be an $m$-dimensional Einstein
manifold satisfying (\ref{warp7i}). Then, the warped product $M=I\times
_{g}\Sigma$ is Einstein with $^{M}Ric=-mc\left\langle\, ,\right\rangle $ and it
is an almost Ricci soliton with potential $f\left(  t\right)  $ and
soliton function $\lambda\left(  t\right)  $ defined in (\ref{warp10}).
\end{proposition}

The next rigidity theorem, in the complete case, shows that basically there are no
further examples when $\left(M,\left\langle\, ,\right\rangle\right)$ is Einstein.

\begin{theorem}
\label{th_classification}Let $\left(  M,\left\langle\,  ,\right\rangle \right)  $
be a complete, connected, Einstein manifold of dimension $m\geq3$ and%
\[
^{M}Ric=-\left(  m-1\right)  c\left\langle\,  ,\right\rangle,\,\,\,\,c\in\mathbb{R}\text{.}%
\]
Assume that $M$ is an almost Ricci soliton, namely, for some
$\lambda\in C^{\infty}\left(  M\right)  $, \ there is a solution
$f\in C^{\infty}\left(  M\right)  $ of the equation%
\begin{equation*}
^{M}Ric+\mathrm{Hess}\left(  f\right)  =\lambda\left(  x\right)
\left\langle\, ,\right\rangle .
\end{equation*}

\noindent
{\rm\bf (a)} If  $c=0$, then $\lambda$ must be constant and the following
possibilities occur:

\begin{description}
\item[(a.1)] If $\lambda=0$ then $M$ is isometric to a cylinder $\mathbb{R}%
\times\Sigma$ over a totally geodesic, Ricci flat hypersurface $\Sigma\subset
M$. Furthermore, $f\left(  t,x\right)  =at+b$, for some constants
$a,b\in\mathbb{R}$.
\item[(a.2)] If $\lambda=\mathrm{const.}\neq0$ then $M$ is isometric to
$\mathbb{R}^{m}$ and%
\begin{equation}
f\left(  x\right)  =\frac{\lambda}{2}\left\vert x\right\vert ^{2}+\left\langle
b,x\right\rangle +c, \label{2}%
\end{equation}
for some $b\in\mathbb{R}^{m}$ and $c\in\mathbb{R}$.
\end{description}

\noindent
{\rm\bf(b)} If $c\neq0$, then either $\lambda$ is constant and the soliton is trivial, or one of the following cases occurs:
\begin{description}
\item[(b.1)] $c\in\mathbb{R}\backslash\left\{  0\right\}  $ and $M$ is a
space-form of constant curvature $-c$. Furthermore%
\begin{equation}
\left\{
\begin{array}
[c]{l}%
\lambda\left(  x\right)  =a\mathrm{cn}_{-c}\left(  r\left(  x\right)  \right)
-\left(  m-1\right)  c\\
f\left(  x\right)  =c^{-1}a\mathrm{cn}_{-c}\left(  r\left(  x\right)  \right)
+b,
\end{array}
\right.  \label{3}%
\end{equation}
for some constants $a,b\in\mathbb{R}$. Here, $r\left(  x\right)  $
denotes the distance from a fixed origin.

\item[(b.2)] $c>0$ and $M$ is isometric to the warped product $\mathbb{R}%
\times_{g}\Sigma$ where%
\[
g\left(  t\right)  =\frac{g^{\prime}\left(  0\right)  }{\sqrt{c}}\sinh\left(
\sqrt{c}t\right)  +g\left(  0\right)  \cosh\left(  \sqrt{c}t\right)  >0,
\]
and $\Sigma\subset M$ is an Einstein hypersurface with
$$^{\Sigma}Ric=-\left( m-2\right) (-g'(0)^2 + c g(o)^2).$$
Furthermore%
\begin{equation}
\left\{
\begin{array}
[c]{l}%
\lambda\left(  t,x\right)  =ag^{\prime}\left(  t\right)  -mc\\
f\left(  t,x\right)  =a\int_{0}^{t}g\left(  s\right)  ds+b,
\end{array}
\right.  \label{3a}%
\end{equation}
for some constants $a,b\in\mathbb{R}$. Here, $t$ is a global
coordinate on $\mathbb{R}$.
\end{description}
\end{theorem}

\begin{proof}
By assumption, with respect to a local orthonormal coframe, we have%
\begin{equation}
f_{ij}=\left(  \left(  m-1\right)  c+\lambda\right)  \delta_{ij}. \label{4}%
\end{equation}
Differentiating both sides and using the commutation rule%
\begin{equation}
f_{ijk}-f_{ikj}=R_{lijk}f_{l}, \label{5}%
\end{equation}
we deduce%
\begin{equation*}
R_{lijk}f_{l}=\lambda_{k}\delta_{ij}-\lambda_{j}\delta_{ik}.
\end{equation*}
Tracing this latter with respect to $i$ and $k$, recalling that $R_{ij}%
=-\left(  m-1\right)  c\delta_{ij},$ and simplifying we conclude that
\begin{equation}
cf_{j}=\lambda_{j}. \label{7}%
\end{equation}
We now distinguish several cases.\medskip

\noindent
\textbf{(a)} Suppose $c=0$, i.e., $M$ is Ricci flat. Then $\lambda_{j}=0$
proving that $\lambda$ is constant. The soliton equation reads%
\begin{equation*}
\mathrm{Hess}\left(  f\right)  =\lambda\left\langle ,\right\rangle .
\end{equation*}

\noindent
\textbf{(a.1)} In case $\lambda=0$, then $f$ is affine. In particular
$\left\vert \nabla f\right\vert $ is constant proving that either $f$ is
constant, and the soliton is trivial, or $f$ has no critical point at all.
Suppose this latter case occurs. Up to rescaling $f$ we can assume that
$\left\vert \nabla f\right\vert =1,$ i.e., $f$ is a function of distance type.
Then, a Cheeger-Gromoll type argument (see (b.2.ii$_{1}$) below for details)
shows that the flow $\phi$ of the vector field $X=\nabla f$ establishes a
Riemannian isometry $\phi:\mathbb{R}\times\Sigma\rightarrow M$, where $\Sigma$
is any of the (totally geodesic) level sets of $f$ and $f$ is a linear
function of $t$. Finally, since $M$ is Ricci flat then also $\Sigma$ must be
Ricci flat. This proves the first part of statement (a) of the Theorem.

\noindent
\textbf{(a.2)} Assume $\lambda\neq0$. Then, it is known that $M$ is isometric
to $\mathbb{R}^{m}$ and $f\left(  x\right)  $ takes the form given in
(\ref{2}). See \cite{T} and the Appendix in \cite{PRiS} for a straightforward proof. The proof of
case (a) is completed.\medskip

\noindent
\textbf{(b)} Suppose $c\neq0$. By (\ref{7}) we have%
\begin{equation}
f=c^{-1}\lambda+d, \label{9}%
\end{equation}
for some constant $d\in\mathbb{R}$. Inserting into (\ref{4}) gives%
\begin{equation*}
\mathrm{Hess}\left(  \lambda\right)  =c\left\{  \left(  m-1\right)
c+\lambda\right\}  \left\langle \, ,\right\rangle .
\end{equation*}

If $\lambda\left(  x\right)  =\mathrm{const.}$, i.e. $M$
is a classical Ricci soliton, then, in view of (\ref{9}), $f$ must be constant
and the soliton is trivial.

Assume then that $\lambda\left(  x\right)  $ is nonconstant. Note that the
function%
\begin{equation}
v\left(  x\right)  =\left(  m-1\right)  c+\lambda\left(  x\right)  \label{11}%
\end{equation}
is a nontrivial solution of%
\begin{equation}
\mathrm{Hess}\left(  v\right)  =cv\left\langle ,\right\rangle . \label{12}%
\end{equation}

\textbf{(b.1.i)} If  $c<0$, then by the classical Obata theorem,
\cite{O}, $M$ is isometric to a spaceform of constant curvature $-c>0$ and%
\begin{equation*}
v\left(  x\right)  =a\cos\left(  \sqrt{-c}r\left(  x\right)  \right)  ,
\end{equation*}
for some constant $a\neq0$. Here, $r\left(  x\right)  $ denotes the distance
function from a fixed origin. It follows that the functions $\lambda\left(
x\right)  $ and $f\left(  x\right)  $ take the form given in \textbf{(b.1)}, (\ref{3}),
 for $c<0$.

It remains to consider the case $c>0$. Two possibilities can occur:

\noindent
\textbf{(b.1.ii)}  The function $v,$ which is a nontrivial
solution of (\ref{12}),  $v$ has  at least one critical point $o\in M$ and, therefore, it is a nontrivial solution of the
problem%
\[
\left\{
\begin{array}
[c]{l}%
\mathrm{Hess}\left(  v\right)  =cv\left\langle ,\right\rangle \\
\left\vert \nabla v\right\vert \left(  o\right)  =0,
\end{array}
\right.
\]
with $c>0$. Thus, for every unit speed geodesic $\gamma$ issuing from $o$, the function
$y=v\circ \gamma$ satisfies the initial value problem
\begin{equation*}
\begin{cases}
y'' = cy&\\
y(0) = v(o),\, y'(0) = \langle \nabla v (o), \dot{\gamma}(0)\rangle,&
\end{cases}
\end{equation*}
and since $v$ is nonconstant, we have must have $v\left(  o\right)  \neq0$.
Using Kanai's version of Obata theorem, \cite{K}, we conclude that $M$ is
isometric to hyperbolic space of constant curvature $-c<0$ and $v\left(
x\right)  =v\left(  o\right)  \cosh\left(  \sqrt{c}r\left(  x\right)  \right)
$ where $r\left(  x\right)  $ is the distance function from $o$. Inserting
this expression into (\ref{11}) and (\ref{9}) completes the proof of case \textbf{(b.1)}

\textbf{(b.2)} The function $v$ has no critical points. A classification of $M$ under
this assumption, and the corresponding form of $v,f,\lambda$, \ can be deduced from
some works by Ishihara and Tashiro, \cite{IT}, and Tashiro,
\cite{T}.
However we provide a concise and complete proof for the sake of completeness.
Let $\ \Sigma=\left\{  v\left(  x\right)  =s\right\}  $ be a non-empty,
smooth, level hypersurface.

Note that, up to multiplying $v$ by a non-zero constant, we can always assume
that either $s=0$ or $s=1$. A computation that uses (\ref{12}) shows that the
integral curves of the complete vector field $X=\nabla v/\left\vert \nabla
v\right\vert $ are unit speed geodesics orthogonal to $\Sigma$. Moreover, the
flow of $X$ gives rise to a smooth map $\phi:\mathbb{R}\times\Sigma\rightarrow
M$ which coincides with the normal exponential map $\exp^{\bot}$ of $\Sigma$.
In particular, $\phi$ is surjective. Evaluating (\ref{12}) along the integral
curve $\phi\left(  t,x\right)  $ issuing from $x\in\Sigma$ we deduce that
$y\left(  t\right)  =v\left(  \phi\left(  t,x\right)  \right)  $ satisfies%
\[
\left\{
\begin{array}
[c]{l}%
y^{\prime\prime}=cy\\
y\left(  0\right)  =s\in\left\{  0,1\right\}  \\
y^{\prime}\left(  0\right)  =\left\vert \nabla v\right\vert \left(  x\right)
\end{array}
\right.
\]
and therefore%
\begin{equation}
v\left(  \phi\left(  t,x\right)  \right)  =\left\vert \nabla v\right\vert
\left(  x\right)  \mathrm{sn}_{-c}\left(  t\right)  +s\mathrm{cn}_{-c}\left(
t\right)  .\label{14}%
\end{equation}
Since%
\begin{equation}
\frac{dv\left(  \phi\left(  t,x\right)  \right)  }{dt}=\left\vert \nabla
v\right\vert \circ\phi\left(  t,x\right)  >0\label{14a}%
\end{equation}
it follows from (\ref{14}) that, necessarily, $c>0$. Moreover, if $s=1$ we
have the further restriction $\left\vert \nabla v\right\vert \left(  x\right)
\geq\sqrt{c}$. The function $v$ is strictly increasing along the geodesic
curves $\phi_{x}\left(  t\right)  $ issuing from $x\in\Sigma$. Whence, it is
easy to conclude that $\phi$ is also injective, hence a diffeomorphism. Since
$M\approx\mathbb{R}\times\Sigma$ is connected, also $\Sigma$ must be
connected. As a consequence, $\left\vert \nabla v\right\vert $ is constant on
$\Sigma$. Indeed, for any smooth curve $\gamma\subset\Sigma$, we have
\begin{align*}
\frac{d}{dt}\left(  \left\vert \nabla v\right\vert \circ\gamma\right)   &
=\mathrm{Hess}\left(  v\right)  \left(  \frac{\nabla v}{\left\vert \nabla
v\right\vert }\circ\gamma,\dot{\gamma}\left(  t\right)  \right)
\\
&  =cv\left(  \gamma\right)  \left\langle X_{\gamma},\dot{\gamma}\left(
t\right)  \right\rangle \nonumber\\
&  =0,\nonumber
\end{align*}
because $\dot{\gamma}\left(  t\right)  \in T\Sigma$ and $X_{\gamma}$ is
orthogonal to $\Sigma$. Therefore $\left\vert \nabla v\right\vert \left(
x\right)  =a\geq\sqrt{c}$, for every $x\in\Sigma$. Using this information into
(\ref{14}) with $c>0$ gives%
\begin{equation*}
v\left(  \phi\left(  t,x\right)  \right)  =\alpha\left(  t\right)
\end{equation*}
where we have set%
\[
\alpha\left(  t\right)  =\frac{a}{\sqrt{c}}\sinh\left(  \sqrt{c}t\right)
+s\cosh\left(  \sqrt{c}t\right)  .
\]
In particular, $\phi$ moves $\Sigma$ onto every other level set of $v$. To
conclude, we show that%
\begin{equation}
\phi^{\ast}\left\langle \,,\right\rangle =dt^{2}+\left(  \alpha^{\prime}\right)
^{2}\left(  t\right)  \left\langle\, ,\right\rangle _{\Sigma_{0}},\label{17}%
\end{equation}
where   $ \left\langle\, ,\right\rangle _{\Sigma}
=\left(  \phi_{0}\right)  ^{\ast }\left\langle\, ,\right\rangle $ denotes the metric induced by $M$ on the
smooth hypersurface $\Sigma$. Indeed, by the above reasonings (or applying Gauss
Lemma) we have%
\begin{equation*}
\phi^{\ast}\left\langle \,,\right\rangle =dt^{2}+\left(  \phi_{t}\right)
^{\ast}\left\langle\, ,\right\rangle .
\end{equation*}
Furthermore, using (\ref{12}), (\ref{14a}) and the definition of the Lie
derivative, we see that, on $T\Sigma_{\phi_{t}}=X_{\phi_{t}}^{\bot}$,%
\[
\frac{d}{dt}\left(  \phi_{t}\right)  ^{\ast}\left\langle \,,\right\rangle
=\frac{2\alpha^{\prime\prime}}{\alpha^{\prime}}\phi_{t}^{\ast}\left\langle\,
,\right\rangle .
\]
Whence, integrating on \thinspace$\lbrack0,t]$ we conclude the validity of
(\ref{17}). Summarizing, we have obtained that, if $v$ has no critical point,
then $\left(  M,\left\langle \,,\right\rangle \right)  $ is isometric to the
warped product manifold%
\begin{equation*}
\left(  \mathbb{R}\times\Sigma,dt^{2}+\alpha^{\prime}\left(  t\right)
^{2}\left\langle \,,\right\rangle _{\Sigma}\right)  ,
\end{equation*}
with $\Sigma$ a smooth hypersurface of $M$. By assumption, $M$ is Einstein
with constant Ricci curvature $-\left(  m-1\right)  c$, therefore $\Sigma$ is Einstein and the
expression of its Ricci curvature follows from Lemma~\ref{lemma_es}.

\textbf{(b.2.ii} To conclude, assume that $v$ possesses at least
one critical point $o\in M$ and, therefore, it is a nontrivial solution of the
problem%
\[
\left\{
\begin{array}
[c]{l}%
\mathrm{Hess}\left(  v\right)  =cv\left\langle\, ,\right\rangle \\
\left\vert \nabla v\right\vert \left(  o\right)  =0,
\end{array}
\right.
\]
with $c>0$. Since $v$ is nonconstant, we have $v\left(  o\right)  \neq0$.
Using Kanai version of Obata theorem, \cite{K}, we conclude that $M$ is
isometric to the hyperbolic space of constant curvature $-c<0$ and $v\left(
x\right)  =v\left(  o\right)  \cosh\left(  \sqrt{c}r\left(  x\right)  \right)
$ where $r\left(  x\right)  $ is the distance function from $o$. Inserting
this expression into (\ref{11}) and (\ref{9}) completes the proof of case (b)
and, hence, of the theorem.
\end{proof}

\begin{example} \label{not-almostsolitons}
\rm{
Let $M$ be any (possibly trivial) quotient of the Riemannian product of
standard spheres $\mathbb{S}^{2}\times \mathbb{S}^{2}$  or a non trivial quotient of
$\mathbb{S}^{m}$. Then, $M$ is Einstein, and according to Theorem \ref{th_classification},
({\bf b.1}) $M$ has no nontrivial almost
Ricci soliton structure.

A similar conclusion holds for possibly  trivial quotients of the Riemannian product of standard
hyperbolic spaces $\mathbb{H}^{2}\times\mathbb{H}^{2}$. Clearly it suffices to consider $\mathbb{H}^{2}\times\mathbb{H}^{2}$ itself. Since  $\mathbb{H}^{2}\times\mathbb{H}^{2}$
is Einstein with $\mathrm{Ric} = -\langle \,,\rangle$, if it had
the structure of a nontrivial almost soliton structure, by Theorem~\ref{th_classification} it would be isometric
to the warped product $\mathbb{R}\times_g\Sigma$ where $\Sigma$ is a 3 dimensional Einstein hypersurface
and $g$ has the form given in the statement of the Theorem. It follows that $\Sigma$ has constant negative curvature,
and, from the expression of the Riemann tensor of a warped product (see e.g., \cite{ONeill}),
$\mathbb{R}\times_g\Sigma \approx \mathbb{H}^{2}\times\mathbb{H}^{2}$ would have strictly negative sectional curvature, which is clearly impossible. Notice that the above reasoning shows that in case {\bf b.2} if $m=4$  and $M$ is simply connected  then $\Sigma$ is a  hyperbolic space.}
\end{example}

Now, suppose that we are given a warped product $M=I\times_{g}\Sigma$ where
$\left(\Sigma, (\,,)\right)$ is an $m$-dimensional Einstein manifold and $0\in I$. If $m\geq 3$, then, for some constant $a$,
\begin{equation*}
^{\Sigma}Ric=-\left(m-1\right)a\left(\, ,\right)_{\Sigma}.
\end{equation*}
According to Lemma \ref{lemma_es} in order that $M$ be Einstein with $^{M}Ric=-mc\left\langle ,\right\rangle$
for some  $c\in \mathbb{R}$, then $g$ must be given by  (\ref{warp7l}) and then
$
cg\left(0\right)^2-g^{\prime}\left(0\right)^2=a.
$
Therefore if (\ref{warp7l}) is not satisfied, then
 $M$ is not Einstein. We consider a function
$f\left(x,t\right)=f\left(t\right)$, so that,  using (\ref{warp7}) and (\ref{warp8})
we see that to give $M=I\times_{g}\Sigma$ the structure of an almost soliton
we need to solve the system
\begin{equation}
\label{1.40}
\left\{
\begin{array}{l}
f^{\prime}\frac{g^{\prime}}{g}
=\lambda+\left(m-1\right)\left(\frac{g^{\prime}}{g}\right)^{2}+\frac{g^{\prime\prime}}{g}
+\frac{\left(m-1\right)a}{g^{2}}\\
f^{\prime\prime}=\lambda+m\frac{g^{\prime\prime}}{g}\\
\end{array}
\right.
\end{equation}
on $I$. Subtracting the first equation from the second we obtain
\begin{equation*}
\left(\frac{f^{\prime}}{g}\right)^{\prime}=\left(m-1\right)\frac{gg^{\prime\prime}-\left(g^{\prime}\right)^{2}-a}{g^{3}}=\left(m-1\right)h\left(t\right)
\end{equation*}
on I,
%
%
and integrating
\begin{equation}\label{1.43}
f\left(t\right)=B+\int^{t}_{0}g\left(s\right)\left[A+\left(m-1\right)
\int^{s}_{0}\frac{g^{\prime\prime}g-\left(g^{\prime}\right)^{2}-a}{g^{3}}dx\right]ds
\end{equation}
for some constants $A,B \in \mathbb R$. Going back to (\ref{1.40}) we then deduce
\begin{equation}\label{1.44}
\lambda\left(t\right)=-\left(m-1\right)\frac{\left(g^{\prime}\right)^{2}+a}{g^{2}}-\frac{g^{\prime\prime}}{g}+g^{\prime}\left[A+\left(m-1\right)\int^{t}_{0}\frac{g^{\prime\prime}g-\left(g^{\prime}\right)^{2}-a}{g^{3}}dx\right].
\end{equation}
Summarizing we have obtained the following new set of examples.
\begin{example}\rm{
Let $M=I\times_{g}\Sigma^{m}$ where $\Sigma^{m}$ is an Einstein manifold satisfying
$^{\Sigma}Ric=-\left(m-1\right)a$ with $a<0$. Then, $M$ supports an almost soliton
structure $f^{\prime}\frac{\partial}{\partial t}$ whith soliton
function $\lambda\left(t\right)$ where $f\left(t\right)$ and $\lambda\left(t\right)$
are defined respectively in (\ref{1.43}) and (\ref{1.44}).}
\end{example}
\begin{remark}
\rm{As observed above , if $g$ does not satisfy (\ref{warp7l}), these almost solitons are not Einstein hence necessarily different from those produced in Proposition \ref{es2} above. We also note that if $\Sigma$ is the
standard $(m-1)$ sphere, and $g$ is defined on $I=[-1,+\infty)$ satisfies $g^{(2k)}(-1) =0$, $g'(-1)=1$
then  we obtain a model manifold in the sense of Greene an Wu (with radial variable $r=t+1$), and the almost soliton structure, which is in general defined only on $(-1, + \infty)$ extends to $[-1, +\infty)$ provided the functions $f$ and $\lambda$ can be smoothly extended in $t=-1$.  We note that  expanding the function $h$ as $t\to -1+$ we obtain that
\[
\ h\left(t\right)\sim \frac{-a-1+o\left((t-1)^{3}\right)}{(t-1)^{3}}.
\]
 Thus $h$ integrable in a neighborhood of $t=-1$ and $f$ and $\lambda$ can be extended to $t=-1$ if and only if $a=-1$.
}
\end{remark}
\section{Some basic formulas}
The aim of this section is to prove some basic formulas for gradient Ricci almost solitons. Some of them are well known for solitons, but we have chosen to reproduce computations here since in our more general setting $\lambda$ is a function and significant extra terms appear along the way. Throughout this section computations are performed with the method of the moving frame in a local orthonormal coframe for the metric $\left\langle \,,\right\rangle$.
\begin{lemma}\label{lemma_2.1}
Let $\left(M,\left\langle \, ,\right\rangle,\nabla f\right)$ be a gradient Ricci almost soliton. Then
\begin{equation}\label{2.1}
\frac{1}{2}\Delta_{f}\left|\nabla f\right|^{2}=\left|Hess\left(f\right)\right|^{2}-\lambda\left|\nabla f\right|^{2}-\left(m-2\right)\left\langle \nabla \lambda,\nabla f\right\rangle.
\end{equation}
\end{lemma}
\begin{proof}
We recall the defining equations
\begin{equation}\label{2.2}
R_{ij}=\lambda\delta_{ij}-f_{ij}.
\end{equation}
Taking covariant derivatives
\begin{equation}\label{2.3}
R_{ij,k}=\lambda_{k}\delta_{ij}-f_{ijk}.
\end{equation}
Tracing with respect to $j$ and $k$
\begin{equation}\label{2.4}
R_{ik,k}=\lambda_{i}-f_{ikk}.
\end{equation}
Next tracing the second Bianchi identities
\[
\ R_{ijkl,s}+R_{ijls,k}+R_{ijsk,l}=0
\]
with respect to $i$ and $s$ we have
\[
\ R_{ijkl,i}=R_{jl,k}-R_{jk,l}
\]
and tracing again with respect to $j$ and $l$
\begin{equation}\label{2.4.1}
2R_{ik,i}=S_{k}\,\,,
\end{equation}
where $S$ denotes the scalar curvature. Using the commutation relations
\[
\ R_{ij,k}=R_{ji,k}
\]
we then deduce
\begin{equation}\label{2.5}
R_{ki,i}=\frac{1}{2}S_{k}
\end{equation}
Using (\ref{5}) and (\ref{2.5}) into (\ref{2.4}) we finally obtain
\begin{equation}\label{2.6}
\frac{1}{2}S_{i}=\lambda_{i}-f_{kki}-f_{t}R_{ti}.
\end{equation}
Now, tracing (\ref{2.3}) with respect to $i$ and $j$ yields
\begin{equation}\label{2.6.1}
S_{i}=m\lambda_{i}-f_{kki}
\end{equation}
so that, substituting into (\ref{2.6}) gives
\begin{equation}\label{2.7}
S_{i}=2\left(m-1\right)\lambda_{i}+2f_{k}R_{ki}.
\end{equation}
In particular, from (\ref{2.7}) we obtain
\begin{equation}\label{2.8}
\left\langle \nabla S,\nabla f\right\rangle=2\left(m-1\right)\left\langle \nabla \lambda,\nabla f\right\rangle+2Ric\left(\nabla f,\nabla f\right).
\end{equation}
Next we recall Bochner formula
\begin{equation}\label{2.9}
\frac{1}{2}\Delta\left|\nabla f\right|^{2}=\left|Hess\left(f\right)\right|^{2}+Ric\left(\nabla f,\nabla f\right)+\left\langle \nabla\Delta f,\nabla f\right\rangle.
\end{equation}
Tracing (\ref{2.2})
\[
\ S=m\lambda-\Delta f
\]
so that
\begin{equation}\label{2.10}
\nabla\Delta f=m\nabla\lambda-\nabla S.
\end{equation}
Inserting (\ref{2.10}) into (\ref{2.9}) and using (\ref{2.8})
\begin{align*}
\ \frac{1}{2}\Delta\left|\nabla f\right|^{2}=&\left|Hess\left(f\right)\right|^{2}+Ric\left(\nabla f,\nabla f\right)+m\left\langle\nabla \lambda,\nabla f \right\rangle-\left\langle \nabla S,\nabla f\right\rangle\\
=&\left|Hess\left(f\right)\right|^{2}-Ric\left(\nabla f,\nabla f\right)-\left(m-2\right)\left\langle \nabla \lambda,\nabla f\right\rangle.
\end{align*}
On the other hand, using
\[
\ \frac{1}{2}\left\langle \nabla\left|\nabla u\right|^{2},X\right\rangle=Hess\left(u\right)\left(\nabla u, X\right)
\]
and (\ref{2.2}) from the above we obtain
\begin{align*}
\frac{1}{2}\Delta_{f}\left|\nabla f\right|^{2}=&\frac{1}{2}\Delta \left|\nabla f\right|^{2}-\frac{1}{2}\left\langle\nabla f,\nabla\left|\nabla f\right|^{2} \right\rangle\\
=&\left|Hess\left(f\right)\right|^{2}-\left(m-2\right)\left\langle \nabla \lambda,\nabla f\right\rangle
\\& \quad -Ric\left(\nabla f,\nabla f\right)-Hess\left(f\right)\left(\nabla f,\nabla f\right)\\
=&\left|Hess\left(f\right)\right|^{2}-\lambda\left|\nabla f\right|^{2}-\left(m-2\right)\left\langle \nabla \lambda,\nabla f\right\rangle,
\end{align*}
that is, (\ref{2.1})
\end{proof}
\begin{corollary}\label{cor_2.11}
\begin{equation}\label{2.12}
\left|\nabla f\right|\Delta_{f}\left|\nabla f\right|\geq -\lambda\left|\nabla f\right|^{2}-\left(m-2\right)\left\langle \nabla \lambda,\nabla f\right\rangle
\end{equation}
\end{corollary}
\begin{proof}
From Kato's inequality
\[
\ \left|Hess\left(f\right)\right|^{2}\geq\left|\nabla\left|\nabla f\right|\right|^{2}
\]
Inserting into (\ref{2.1}) we obtain (\ref{2.12}).
\end{proof}
We let $S$ denote the scalar curvature and $W$ the Weyl tensor of $\left(M,\left\langle ,\right\rangle\right)$.
\begin{lemma}\label{lemma_2.2}
Let $\left(M,\left\langle \, ,\right\rangle,\nabla f\right)$ be a gradient Ricci almost soliton of dimension $m\geq 3$. Then
\begin{align}\label{2.13}
\Delta_{f}R_{ik}&=\Delta\lambda\delta_{ik}+\left(m-2\right)\lambda_{ik}
+2\lambda R_{ik}
-\frac{2}{m-2}\left(\left|Ric\right|^{2}
 -\frac{S^{2}}{m-1}\right)\delta_{ik}\nonumber
 \\ &-\frac{2m}{\left(m-1\right)\left(m-2\right)}SR_{ik}
+\frac{4}{m-2}R_{is}R_{sk}-2W_{ijks}R_{sj}.
\end{align}
Therefore, tracing with respect to $i$ and $k$
\begin{equation}\label{2.14}
\frac{1}{2}\Delta_{f}S=\lambda S-\left|Ric\right|^{2}+\left(m-1\right)\Delta \lambda.
\end{equation}
\end{lemma}
\begin{remark}\rm{
Note that for (\ref{2.14}) we do not need the restriction $m\geq 3$.
Indeed (\ref{2.14}) can also be obtained by tracing (\ref{2.18})
below for which it is not required $m\geq 3$.}
\end{remark}
\begin{proof}
It follows from (\ref{2.3}) and the commutations relations
$f_{ijk}-f_{ikj} = R_{lijk}f_l$ that
\begin{equation}\label{2.15}
R_{ik,j}-R_{jk,i}=f_{s}R_{ijks}+\lambda_{j}\delta_{ki}-\lambda_{i}\delta_{kj},
\end{equation}
and taking covariant derivatives  we obtain the commutation relations
\begin{equation}\label{2.16}
R_{ik,jt}-R_{jk,it}=f_{st}R_{ijks}+f_{s}R_{ijks,t}+\lambda_{jt}\delta_{ki}-\lambda_{it}\delta_{kj}.
\end{equation}
Also, from the commutation relations  for the second covariant derivative of $R_{ik}$
we have
\[
\ R_{ij,kl}- R_{ij,lk}=R_{it}R_{tjkl}+R_{jt} R_{tikl},
\]
whence, contracting we obtain
\begin{equation*}
R_{jk,ij}=R_{jk,ji}+R_{ji,js}R_{sk}+R_{jiks}R_{sj}.
\end{equation*}
We now use (\ref{2.16}) to obtain
\[
\ \Delta R_{ik}=R_{ik,jj}=R_{jk,ij}+f_{s}R_{ijks,j}+f_{sj}R_{ijks}
+\Delta\lambda\delta_{ki}-\lambda_{ik}.
\]
On the other hand, from the second Bianchi identities we have
\[
\ f_{s}R_{ijks,j}=R_{ik,s}f_{s}-R_{is,k}f_{s}
\]
and inserting this into the above identity yields
\begin{multline*}
\ \Delta R_{ik}=f_{sj}R_{ijks}-f_{s}R_{is,k}+f_{s}R_{ik,s}\\+R_{jk,ji}+R_{sk}R_{is}
+R_{sj}R_{jiks}+\Delta\lambda\delta_{ik}-\lambda_{ik}.
\end{multline*}
Hence, from (\ref{2.4.1}) and (\ref{2.2})
\begin{multline}\label{2.18}
\Delta R_{ik}=\frac{1}{2}S_{ki}+\lambda R_{ik}+R_{sk}R_{is}+\Delta\lambda\delta_{ik}-\lambda_{ik}\\
-2R_{ijks}R_{sj}-R_{is,k}f_{s}+R_{ik,s}f_{s}.
\end{multline}
We shall now deal with the sum
\begin{equation}\label{2.19}
Z=\frac{1}{2}S_{ki}+R_{sk}R_{is}-f_{s}R_{is,k}.
\end{equation}
Towards this aim we first observe that taking covariant derivative of (\ref{2.7}) we have
\begin{equation*}
\frac{1}{2}S_{ik}=f_{jk}R_{ij}+f_{j}R_{ij,k}+\left(m-1\right)\lambda_{ik}.
\end{equation*}
Substituting this into (\ref{2.19}) we obtain
\[
\ Z=\left(m-1\right)\lambda_{ki}+R_{kj}\left(f_{ji}+R_{ji}\right)+f_{s}\left(R_{ks,i}-R_{is,k}\right)
\]
and using the almost soliton equation (\ref{2.2}), (\ref{2.15}), and the fact that $f_{t}f_{s}R_{kits}=0$ because of the symmetries of the curvature tensor,
\[
\ Z=\left(m-1\right)\lambda_{ki}+\lambda R_{ki}+\lambda_{i}f_{k}-\lambda_{k}f_{i},
\]
Substituing into (\ref{2.18}) we therefore obtain
\begin{align}\label{2.21}
\Delta_{f}R_{ik}=&\Delta R_{ik}-f_{s}R_{ik,s}\\
=&2\lambda R_{ik}+\Delta\lambda\delta_{ik}+\left(m-2\right)\lambda_{ik}-2R_{ijks}R_{sj}+
\lambda_{i}f_{k}-\lambda_{k}f_{i}.\nonumber
\end{align}
Notice that the all the terms in the above formula are symmetric in $i,k$ with the exception of
$\lambda_{i}f_{k}-\lambda_{k}f_{i}$ which is skew symmetric.
So we must have $\lambda_{i}f_{k}-\lambda_{k}f_{i}=0$.
The conclusion now follows recalling the decomposition of the curvature tensor into its irreducible components.
\begin{align*} R_{ijks}=&W_{ijks}+\frac{1}{m-2}\left(R_{ik}\delta_{js}
-R_{is}\delta_{jk}+R_{js}\delta_{ik}-R_{jk}\delta_{is}\right)
\\&-\frac{S}{\left(m-1\right)\left(m-2\right)}\left(\delta_{ik}\delta_{js}-\delta_{is}\delta_{jk}\right).\nonumber
\end{align*}
Substituting into (\ref{2.21}) we obtain (\ref{2.13}).
\end{proof}
\begin{remark}
{\rm
In the course of the above proof we have obtained that if $M$ is a gradient almost Ricci soliton
with potential $f$ and soliton function $\lambda$ then $\lambda_{i}f_{k}-\lambda_{k}f_{i}=0$, that is,
$df\wedge d\lambda =0$. It follows that  $\lambda$ is a function of $f$ in the set $\{x: df\neq 0\}$.
}
\end{remark}

\begin{corollary}\label{2.22}
Let $\left(M,\left\langle\, ,\right\rangle\,\nabla f\right)$ be a conformally flat gradient Ricci almost soliton of dimension $m\geq 3$. Then
\begin{align*}
\Delta_{f}R_{ik}&=\Delta\lambda\delta_{ik}+  \left(m-2\right)\lambda_{ik}
+2\lambda R_{ik} -\frac{2}{m-2}\left(\left|Ric\right|^{2}
-\frac{S^{2}}{m-1}\right)\delta_{ik} \nonumber
\\&
-\frac{2m}{\left(m-1\right)\left(m-2\right)}S R_{ik}
+\frac{4}{m-2}R_{is}R_{sk}.
\end{align*}
\end{corollary}

\begin{corollary}\label{2.24}
Let $\left(M,\left\langle\,  ,\right\rangle,\nabla f\right)$ be as in Corollary \ref{2.22} and let $T=Ric-\frac{S}{m}\left\langle \,,\right\rangle$ be the trace free Ricci tensor. Then
\begin{align}\label{2.25}
\frac{1}{2}\Delta_{f}\left|T\right|^{2}=&\left|\nabla T\right|^{2}+2\left(\lambda - S\frac{m-2}{m\left(m-1\right)}\right)\left|T\right|^{2}\\&+\left(m-2\right)\left\langle Hess\left(\lambda\right), T\right\rangle+\frac{4}{m-2}tr\left(T^{3}\right)\nonumber,
\end{align}
with $T^{3}=T\circ T\circ T$. In particular, using Okumura's Lemma
\begin{align}\label{2.26}
\frac{1}{2}\Delta_{f}\left|T\right|^{2}\geq& 2\left(\lambda-S\frac{m-2}{m\left(m-1\right)}\right)\left|T\right|^{2}
-\frac{4}{\sqrt{m\left(m-1\right)}}\left|T\right|^{3}\\&+\left(m-2\right)\left\langle Hess\left(\lambda\right), T\right\rangle.\nonumber
\end{align}
\end{corollary}
\begin{proof}
We compute
\begin{equation*}
\begin{split}
\Delta_{f}\left|T\right|^{2}&=2\left|\nabla T\right|^{2}+2\left\langle T,\Delta T\right\rangle-\left\langle \nabla f, \nabla \left|T\right|^{2}\right\rangle\\
&= 2T_{ik,l}T_{ik,l} +2 T_{ik}\Delta_fT_{ik}.
\end{split}
\end{equation*}
Using equation (\ref{2.14}) and the definition of $T$, we have
\begin{align*}
\Delta_{f}T_{ik}&=\Delta_{f}R_{ik}-\frac{1}{m}\delta_{ik}\Delta_{f}S=
\Delta\lambda\delta_{ik}+\left(m-2\right)\lambda_{ik}
+2\lambda R_{ik}
\\ &
-\frac{2}{m-2}\left|Ric\right|^{2}\delta_{ik}
+\frac{2}{\left(m-2\right)\left(m-1\right)}S^{2}\delta_{ik}
-\frac{2m}{\left(m-1\right)\left(m-2\right)}SR_{ik}\\
&+\frac{4}{m-2}R_{is}R_{sk}-\frac{2}{m}\lambda S\delta_{ik}+\frac{2}{m}\left|Ric\right|^{2}\delta_{ik}-\frac{2}{m}\left(m-1\right)\Delta\lambda\delta_{ik}\\
& =-\frac{m-2}{m}\delta_{ik}\Delta\lambda+\left(m-2\right)\lambda_{ik} +2\lambda T_{ik}-\frac{2mS}{\left(m-1\right)\left(m-2\right)}T_{ik}
\\&-\frac{4}{m\left(m-2\right)}\delta_{ik}\left|Ric\right|^{2}
+\frac{4}{m-2} \bigl( T_{is}T_{sk} + \frac{S^2}{m^2} \delta_{ik} +\frac{2S}{m}T_{ik}\bigr).
\end{align*}
Thus, recalling that $T$ is trace free, we obtain
\begin{align*}
T_{ik}\Delta_f T_{ik}&= 2\lambda\left|T\right|^{2}+\left(m-2\right)\lambda_{ik}T_{ik}
\\&
-\frac{2 (m-2)  S}{m \left(m-1\right)}\left|T\right|^{2}
+ \frac{4}{m-2} \mathrm{tr}(T^3),
\end{align*}
and (\ref{2.25}) follows. Inequality (\ref{2.26}) follows immediately, since by Okumura's Lemma, \cite{Ok},
\[
\ tr\left(T^{3}\right)\geq -\frac{m-2}{\sqrt{m\left(m-1\right)}}\left|T\right|^{3}.
\]
\end{proof}

\section{Volume comparison results}
In order to prove Theorem \ref{A} we need two auxiliary results. The first is an improvement of Theorem 1.2 (a) of Wei and Wylie \cite{WW}, where they assume $\theta$ and $G$ below to be constant.
\begin{theorem}\label{3.1}
Let $\left(M,\left\langle\, ,\right\rangle,e^{-f}d\mathrm{vol}\right)$ be a complete weighted manifold such that
\begin{equation}\label{3.2}
\left\langle \nabla r,\nabla f\right\rangle\geq -\theta\left(r\right),
\end{equation}
for some non-decreasing function $\theta\in C^{0}\left(\mathbb{R}_{0}^{+}\right)$. Assume
\begin{equation}\label{3.3}
Ric_{f}\geq-\left(m-1\right)G\left(r\right)\left\langle\, ,\right\rangle
\end{equation}
for a smooth positive function $G$ on $\mathbb{R}^{+}_{0}$, even at the origin. Let $g$ be a solution on $\mathbb{R}^{+}_{0}$ of
\begin{equation}\label{3.4}
\left\{
\begin{array}{l}
g^{\prime\prime}-Gg\geq0\\
g\left(0\right)=0, \,\,\,g^{\prime}\left(0\right)\geq 1\\
\end{array}
\right.
\end{equation}
Then there exists a constants $D>0$ such that $\forall r\geq 0$
\begin{equation}\label{3.5}
vol_{f}\left(B_{r}\right)\leq D\int^{r}_{0}g^{m-1}\left(t\right)e^{\int^{t}_{0}\theta\left(s\right)ds}dt.
\end{equation}
\end{theorem}
\begin{proof}
Let $h$ be the solution on $\mathbb{R}^{+}_{0}$ of the Cauchy problem
\begin{equation}\label{3.7}
\left\{
\begin{array}{l}
h^{\prime\prime}-Gh=0\\
h\left(0\right)=0, \,\,\,h^{\prime}\left(0\right)= 1\\
\end{array}
\right.
\end{equation}
Note that $h>0$ on $\mathbb{R}^{+}$ since $G\geq0$. Fix $x\in M\setminus\left(cut\left(o\right)\cup\left\{o\right\}\right)$ and let $\gamma:\left[0,l\right]\rightarrow M$, $l=length\left(\gamma\right)$, be a minimizing geodesic with $\gamma\left(0\right)=o$, $\gamma\left(l\right)=x$. Note that $G\left(r\circ\gamma\right)\left(t\right)=G\left(t\right)$. From Bochner formula applied to the distance function $r$ we have
\begin{equation}\label{3.7a}
\ 0=\left|Hess\left(r\right)\right|^{2}+\left\langle \nabla r,\nabla\Delta r\right\rangle+ Ric\left(\nabla r, \nabla r\right)
\end{equation}
so that, using the Schwarz inequality, it follows that  the function $\varphi\left(t\right)=\left(\Delta r\right)\circ \gamma\left(t\right)$, $t\in\left(0,l\right]$, satisfies the Riccati inequality
\begin{equation}\label{3.8}
\varphi^{\prime}+\frac{1}{m-1}\varphi^{2}\leq -Ric\left(\nabla r\circ\gamma,\nabla r\circ\gamma\right)
\end{equation}
on $\left(0,l\right]$. With $h$ as in (\ref{3.7}) and using the definition of $Ric_{f}$, (\ref{3.3}) and (\ref{3.8}) we compute
\begin{align*}
\left(h^{2}\varphi\right)^{\prime}=&2hh^{\prime}\varphi+h^{2}\varphi^{\prime}\\
\leq&2hh^{\prime}\varphi-\frac{h^{2}\varphi^{2}}{m-1}+\left(m-1\right)G\left(t\right)h^{2}+ Hess\left(f\right)\left(\nabla r\circ\gamma,\nabla r\circ\gamma\right)h^{2}\\
=&-\left(\frac{h\varphi}{\sqrt{m-1}}-\sqrt{m-1}h^{\prime}\right)^{2}+\left(m-1\right)\left(h^{\prime}\right)^{2}+\left(m-1\right)G\left(t\right)h^{2}\\
&+h^{2}\left(f\circ\gamma\right)^{\prime\prime}.
\end{align*}
We let
\[
\ \varphi_{G}\left(t\right)=\left(m-1\right)\frac{h^{\prime}}{h}\left(t\right)
\]
so that, using (\ref{3.7})
\[
\ \left(h^{2}\varphi_{G}\right)^{\prime}=\left(m-1\right)\left(h^{\prime}\right)^{2}+\left(m-1\right)G\left(t\right)h^{2}.
\]
Inserting into the above inequality we obtain
\begin{equation}\label{3.9}
\left(h^{2}\varphi\right)^{\prime}\leq\left(h^{2}\varphi_{G}\right)^{\prime}+h^{2}\left(f\circ\gamma\right)^{\prime\prime}
\end{equation}
Integrating (\ref{3.9}) on $\left[0,r\right]$ and using (\ref{3.7}) yields
\begin{equation}\label{3.10}
h^{2}\left(r\right)\varphi\left(r\right)\leq h^{2}\left(r\right)\varphi_{G}\left(r\right)+\int^{r}_{0}h^{2}\left(f\circ\gamma\right)^{\prime\prime}.
\end{equation}
Next we recall that
\begin{equation}
\label{3.9bis}
\varphi_{f}=\left(\Delta_{f}r\right)\circ\gamma
=\left(\Delta r\right)\circ\gamma-\left\langle \nabla f,\nabla r\right\rangle\circ\gamma=\varphi-\left(f\circ\gamma\right)^{\prime}
\end{equation}
Thus, using (\ref{3.10}), (\ref{3.7}) and integrating by parts we compute
\begin{align*}
h^{2}\varphi_{f}\leq&h^{2}\varphi_{G}-h^{2}\left(f\circ\gamma\right)^{\prime}+\int^{r}_{0}h^{2}\left(f\circ\gamma\right)^{\prime\prime}dt\\
=&h^{2}\varphi_{G}-h^{2}\left(f\circ\gamma\right)^{\prime}+\left(h^{2}\left(f\circ\gamma\right)^{\prime}\left.\right)\right|^{r}_{0}-\int^{r}_{0}\left(h^{2}\right)^{\prime}\left(f\circ\gamma\right)^{\prime}dt\\
=&h^{2}\varphi_{G}-\int^{r}_{0}\left(h^{2}\right)^{\prime}\left(f\circ\gamma\right)^{\prime}dt,
\end{align*}
that is,
\begin{equation}\label{3.11}
h^{2}\varphi_{f}\leq h^{2}\varphi_{G}-\int^{r}_{0}\left(h^{2}\right)^{\prime}\left(f\circ\gamma\right)^{\prime}dt
\end{equation}
on $\left(0,l\right]$. We observe that, because of (\ref{3.7}) and $G\geq0$, $\left(h^{2}\right)^{\prime}=2hh^{\prime}\geq 0$
so that, using (\ref{3.2}), (\ref{3.7}) and the monotonicity of $\theta$, (\ref{3.11}) yields
\[
\ h^{2}\varphi_{f}\leq h^{2}\varphi_{G}+\theta\left(r\right)h^{2}
\]
on $\left(0,l\right]$, and
\begin{equation*}
\varphi_{f}\leq\varphi_{G}+\theta\left(r\right)
\end{equation*}
on $\left(0,l\right]$. In particular
\begin{equation}\label{3.13}
\Delta_{f}r\left(x\right)\leq\left(m-1\right)\frac{h^{\prime}\left(r\left(x\right)\right)}{h\left(r\left(x\right)\right)}+ \theta\left(r\left(x\right)\right)
\end{equation}
on $M\setminus\left(\left\{0\right\}\cup cut\left(o\right)\right)$.
Proceeding as in Theorem 2.4 of \cite{PRS-Progress} one shows that (\ref{3.13})
holds weakly on all of $M$ and reasoning as in Theorem 2.14 of \cite{PRS-Progress}
one shows that
\begin{equation}\label{3.13a}
\ vol_{f}\left(\partial B_{r}\right)\leq Dh\left(r\right)^{m-1}e^{\int^{r}_{0}\theta\left(t\right)dt}
\end{equation}
for some constant $D>0$. Integrating over $\left[0,r\right]$ and using the co-area formula we get
\begin{equation}\label{3.14}
vol_{f}\left(B_{r}\right)\leq D\int^{r}_{0}h\left(t\right)^{m-1}e^{\int^{t}_{0}\theta\left(s\right)ds}dt.
\end{equation}
Since $g$ in (\ref{3.4}) is a subsolution of (\ref{3.7}) it follows, by Lemma 2.1 in \cite{PRS-Progress}, that $h\leq g$ on $\mathbb{R}^{+}_{0}$ so that (\ref{3.14}) immediately implies (\ref{3.5})
\end{proof}
A second estimate on $\varphi_{f}$ can also be derived, replacing assumption (\ref{3.2}) with
\begin{equation}\label{3.15}
\xi\left(r\right)\leq f\leq \omega\left(r\right),
\end{equation}
for some functions $\omega,\xi\in C^{1}\left(\mathbb{R}^{+}_{0}\right)$ with $\omega$ non decreasing and such that $\xi^{\prime}\left(r\right)\leq\omega^{\prime}\left(r\right)$ .\\
Towards this aim we integrate (\ref{3.11}) again by parts to obtain
\[
\ h^{2}\varphi_{f}\leq h^{2}\varphi_{G}-\left.\left[\left(h^{2}\right)^{\prime}\left(f\circ\gamma\right)\right]\right|^{r}_{0}+\int^{r}_{0}\left(h^{2}\right)^{\prime\prime}\left(f\circ\gamma\right)dt.
\]
Now, using (\ref{3.7}),
\[
\ \left(h^{2}\right)^{\prime\prime}=2\left(h^{\prime}\right)^{2}+2Gh^{2}\geq0,
\]
because of the sign of $G$. Thus using (\ref{3.15}), (\ref{3.7}) and the fact that $\omega$ is non-decreasing, from the above we obtain
\begin{align*}
h^{2}\varphi_{f}\leq& h^{2}\varphi_{G}-\left(h^{2}\right)^{\prime}\left(f\circ\gamma\left.\right)\right|^{r}_{0}+\omega\left(r\right)\left.\left(h^{2}\right)^{\prime}\right|^{r}_{0}\\
\leq&h^{2}\varphi_{G}-\left(h^{2}\right)^{\prime}\left(r\right)\left(f\circ\gamma\right)\left(r\right)+\left(h^{2}\right)^{\prime}\left(r\right)\omega\left(r\right)\\
\leq&h^{2}\varphi_{G}+\left(h^{2}\right)^{\prime}\left(r\right)\left[\omega\left(r\right)-\left(f\circ\gamma\right)\left(r\right)\right].
\end{align*}
No
\[
\ \left(h^{2}\right)^{\prime}=2hh^{\prime}=\frac{2}{m-1}h^{2}\left(m-1\right)\frac{h^{\prime}}{h}=\frac{2}{m-1}h^{2}\varphi_{G},\,\,r>0
\]
so that the above inequality may be rewritten  as
\[
\ h^{2}\varphi_{f}\leq h^{2}\left(1+\frac{2}{m-1}\left(\omega\left(r\right)-\left(f\circ\gamma\right)\left(r\right)\right)\right)\varphi_{G},\,\,r>0
\]
and using (\ref{3.15})
\begin{equation*}
\varphi_{f}\leq\left(1+\frac{2}{m-1}\left(\omega\left(r\right)-\xi\left(r\right)\right)\right)\varphi_{G},\,\,r>0.
\end{equation*}
Let $\tilde{\omega}\left(r\right)\geq\omega\left(r\right)-\xi\left(r\right)\geq0$.
Similarly to what we did in Theorem \ref{3.1} we arrive at (\ref{3.13a}),
where $\theta\left(t\right)$ is now substituted by
$\frac{2}{m-1}\tilde{\omega}\left(t\right)\varphi_{G}\left(t\right)$.
Thus we need to estimate
$e^{\int_{r_{0}}^{r}2\tilde{\omega}\left(t\right)\frac{h^{\prime}}{h}}$.
\begin{align*}
&\int^{r}_{r_{0}}\frac{2}{m-1}\tilde{\omega}\left(t\right)\frac{h^{\prime}}{h}\\
=&\frac{2}{m-1}\tilde{\omega}\left(r\right)\log h^{m-1}\left(r\right)-\frac{2}{m-1}\tilde{\omega}\left(r_{0}\right)\log h^{m-1}\left(r_{0}\right)\\&-\int^{r}_{r_{0}}\frac{2}{m-1}\tilde{\omega}^{\prime}\left(t\right)\log h^{m-1}\left(t\right)dt.
\end{align*}
Now, by (\ref{3.7}), $h\left(t\right)\nearrow +\infty$ as $t\rightarrow+\infty$. Choose $r_{0}$ sufficiently large that $h\left(r_{0}\right)\geq1$. Since $\tilde{\omega}^{\prime}\geq0$
\begin{equation*}
\int^{r}_{r_{0}}\frac{2}{m-1}\tilde{\omega}\left(t\right)\left(\log h^{m-1}\right)^{\prime}dt\leq\log\left(h\left(r\right)\right)^{2\tilde{\omega}\left(r\right)}-A,
\end{equation*}
and
\[
\ e^{\int^{r}_{r_{0}}\frac{2}{m-1}\tilde{\omega}\left(t\right)\varphi_{G}}\leq h\left(r\right)^{2\tilde{\omega}\left(r\right)}e^{-A}.
\]
Hence, from (\ref{3.13a}),
\[
\ vol_{f}\left(\partial B_{r}\right)\leq Dh\left(r\right)^{m-1+2\tilde{\omega}\left(r\right)}
\]
Since $h\leq g$ we have thus proven the following result, which improves on Theorem 1.2 (b) of Wei and Wylie \cite{WW}.
\begin{theorem}
Let $\left(M,\left\langle\, ,\right\rangle,e^{-f}d\mathrm{vol}\right)$ be a complete weighted manifold such that
\[
\ \xi\left(r\right)\leq f\leq \omega\left(r\right)
\]
for some functions $\omega, \xi\in C^{1}\left(\mathbb{R}^{+}_{0}\right)$ with $\omega$ non decreasing and such that $\xi^{\prime}\left(r\right)\leq\omega^{\prime}\left(r\right)$. Assume
\[
\ Ric_{f}\geq -\left(m-1\right)G\left(r\right)\left\langle ,\right\rangle
\]
for a smooth positive function $G$ on $\mathbb{R}^{+}_{0}$, even at the origin.\\
Let $\tilde{\omega}\left(r\right)=\omega\left(r\right)-\xi\left(r\right)$ and $g$ be a solution on $\mathbb{R}^{+}_{0}$ of
\begin{equation*}
\left\{\begin{array}{l}
g^{\prime\prime}-Gg\geq 0\\
g\left(0\right)=0,\,\,g^{\prime}\left(0\right)\geq 1\\
\end{array}\right.
\end{equation*}
Then there exist constants $C$,$B>0$ such that, $\forall r\geq
r_{0}>0$,
\[
\ vol_{f}\left(B_{r}\right)\leq C+B\int^{r}_{r_{0}}g\left(t\right)^{\left(m-1\right)+2\tilde{\omega}\left(t\right)}dt.
\]
\end{theorem}

We end this discussion
with the following simple proposition which marginally extends a previous result by
Wei and Wylie, \cite{WW}, and which will be used in the proof of Theorem~\ref{B}.

\begin{proposition}
\label{prop_volest1} Let $(M,\langle \,,\rangle, e^{-f})$ be a
weighted manifold and assume that
\begin{equation*}
 \mathrm{Ric}_f \geq D (1+r)^{-\mu}.
\end{equation*}
\begin{itemize}
\item[(i)] If $D>0$ and $0\leq \mu\leq 1$, then there exist constants $C_j$
such that  for every $r>2$,
\begin{equation*}
\mathrm{vol}_f (\partial B_r)\leq
\begin{cases}
C_1 e^{-C_2 r\log (1+r)} &\text{if } \mu = 1\\
C_1 e^{-C_2 r^{2-\mu}}&\text{if } 0\leq \mu <1
\end{cases}
\text{\,\, and \,\,} \mathrm{vol}_f (B_r)\leq C_3
\end{equation*}
\item[(ii)] If $D=0$ then there exist constants $C_j$ such that for every $r>2$,
\begin{equation*}
\mathrm{vol}_f (\partial B_r)\leq C_1 e^r \text{\,\, and \,\, }
\mathrm{vol}_f (B_r)\leq C_2 e^r.
\end{equation*}
\item[(iii)] If $D<0$ then there exist constants $C_j$ such that for every $r>2$,
\begin{equation*}
\mathrm{vol}_f (\partial B_r)\leq
\begin{cases}
C_1 e^{C_2 r} &\text{if } \mu>1\\
C_1 e^{C_2 r\log r} &\text{if } \mu = 1\\
C_1 e^{C_2 r^{2-\mu}} &\text{if } 0\leq \mu<1
\end{cases}
\end{equation*}
and
\begin{equation*}
\mathrm{vol}_f (B_r)\leq
\begin{cases}
C_3 e^{C_2r} &\text{if } \mu >1\\
C_3 (\log r)^{-1} e^{C_2 r\log r}  &\text{if } \mu=1\\
C_3 r^{\mu-1} e^{C_2 r^{2-\mu}}  &\text{if }  0\leq \mu <1.
\end{cases}
\end{equation*}
\end{itemize}
\end{proposition}
\begin{proof}
Maintaining the notation introduced above, it follows from (\ref{3.8}), (\ref{3.9bis})
and the definition of $\mathrm{Ric}_f$ that if $\varphi_f = \Delta_f r \circ
\gamma$ then
\begin{equation*}
\varphi_f ' = \varphi - (f\circ \gamma)'' \leq - \frac {\varphi
^2}{m-1} -\mathrm{Ric}(\dot\gamma,\dot\gamma) - \mathrm{Hess}
f(\dot\gamma,\dot\gamma) \leq -\mathrm{Ric_f}
(\dot\gamma,\dot\gamma).
\end{equation*}
Thus, if we assume that $\mathrm{Ric_f}\geq \theta(r(x))$ and that the ball  $B_\epsilon$
is contained in the domain of the normal  coordinates at $o$,
setting $C=\max_{\partial B_\epsilon} \Delta_f r $ and
integrating between $\epsilon$ and $r(x)$ we obtain
\begin{equation*}
\Delta_f r (x) \leq  C  + \int_{\epsilon}^{r(x)} \theta (t) dt
\end{equation*}
pointwise in the $M\setminus (B_\epsilon\cup \mathrm{cut}(o))$
and weakly on  $M\setminus B_\epsilon$. From this, arguing as in
\cite{PRS-Progress} Theorem 2.4 we deduce that
\begin{equation*}
\mathrm{vol}_f (\partial B_r) \leq e^{C(r-r_o)+\int_{r_o}^r (\int_\epsilon^t \theta(s) ds) t}
\mathrm{vol}_f (\partial
B_{r_o}).
\end{equation*}
The conclusion now follows estimating the integral on the right hand
side for $\theta = D(1+r)^{-\mu}$.
\end{proof}

The second ingredient we shall need in the proof of Theorem \ref{A} is the following version of Theorem 5.2 in \cite{MRS}.
\begin{theorem}\label{prop_3.1}
Let $\left(M,\left\langle\, ,\right\rangle ,e^{-f}d\mathrm{vol}\right)$ be a complete weighted manifold. Given $\sigma,\mu\in\mathbb{R}$, let $\nu=\mu+2\left(\sigma-1\right)$ and assume that $\sigma\geq0$, $\sigma-\nu>0$. Let $u\in C^{1}\left(M\right)$ be a function such that
\begin{equation*}
\hat{u}=\limsup_{r\left(x\right)\rightarrow +\infty}\frac{u\left(x\right)}{r\left(x\right)^{\sigma}}<+\infty
\end{equation*}
and suppose that
\begin{equation}\label{3.16}
\liminf_{r\rightarrow+\infty}\frac{\log vol_{f}\left(B_{r}\right)}{r^{\sigma-\nu}}=d_{0}<+\infty.
\end{equation}
Then given $\gamma\in \mathbb{R}$ such that
\[
\ \Omega_{\gamma}=\left\{x\in M \,:\, u\left(x\right)>\gamma\right\}\neq\emptyset
\]
we have
\[
\ \inf_{\Omega_{\gamma}}\left(1+r\left(x\right)\right)^{\mu}\Delta_{f}u\leq C\max\left\{\hat{u},0\right\}
\]
with
\begin{equation*}
 C=\left\{\begin{array}{lll}
0&\textrm{if}& \sigma=0\\
d_{0}\left(\sigma-\nu\right)^{2}&\textrm{if}& 0<\nu<\sigma\\
d_{0}\sigma\left(\sigma-\nu\right)&\textrm{if}& \sigma>0$, $\nu\geq\sigma.\\
\end{array}\right.
\end{equation*}
\end{theorem}
We are now ready to give a
\begin{proof}(of Theorem \ref{A}). First of all from Lemma \ref{lemma_2.1} and assumption (\ref{0.9}) we know that $\left|\nabla f\right|^{2}$ satisfies the differential inequality
\begin{equation}\label{3.17}
\Delta_{f}\left|\nabla f\right|^{2}\geq -2\lambda\left|\nabla f\right|^{2}
\end{equation}
on $M$. Furthermore, from (\ref{0.7}) we deduce
\begin{equation*}
\left\langle \nabla r,\nabla f\right\rangle\geq-a\left(1+r\right)^{\frac{\sigma}{2}},
\end{equation*}
for some constant $a>0$. Using (\ref{0.8}) we apply Theorem
\ref{3.1} with the choice
$\theta\left(r\right)=a\left(1+r\right)^{\frac{\sigma}{2}}$ to
obtain
\[
\ vol_{f}\left(B_{r}\right)\leq D\int^{r}_{0}g\left(t\right)^{m-1}e^{\frac{2a}{\left(\sigma+2\right)\left(m-1\right)}\left(1+t\right)^{\frac{\sigma+2}{2}}}dt
\]
for some constant $B>0$ and where $g$ solves (\ref{3.7}) with
$G\left(r\right)=B^{2}\left(1+r^{2}\right)^{\frac{\alpha}{2}}$. By
Proposition 2.11 of \cite{PRS-Progress} it follows that, for $r>>1$,
$$
g(r) \leq C_1 \exp(C_2 e^{\frac{\alpha+2}2}),
$$
for some constant $C_1,C_2$.
Thus, a simple computation shows that
\[
\frac{\log \mathrm{vol}_{f}\left(B_{r}\right)}{r^{2-\mu-\sigma}}\leq C\left(r^{\mu+\sigma-1+\frac{\alpha}{2}}+r^{\mu-1+\frac{3}{2}\sigma}\right)
\]
for $r>>1$ and some constant $C>0$. Using (\ref{0.6}) we see that assumption (\ref{3.16})
of Proposition \ref{prop_3.1} is satisfied  with $\nu=\mu+2\left(\sigma-1\right)$
so that $\sigma-\nu=2-\mu-\sigma$
On the other hand, from (\ref{3.17}) and (\ref{0.8}) we have
\begin{equation}\label{3.19}
\left(1+r\left(x\right)\right)^{\mu}\Delta_{f}\left|\nabla f\right|^{2}\geq H\left|\nabla f\right|^{2}
\end{equation}
for some appropriate constant $H>0$. Assume that $\left|\nabla
f\right|$ is different from $0$ and choose $\gamma>0$ so that
\[
\ \Omega_{\gamma}=\left\{x\in M:\left|\nabla f\right|^{2}>\gamma\right\}\neq\emptyset.
\]
From (\ref{0.7}), (\ref{3.19}) and Theorem \ref{prop_3.1} we immediately obtain a contradiction.
\end{proof}

\section{A weighted Omori-Yau maximum principle}

Theorem \ref{prop_3.1} stated in the previous section represents a refined and generalized version of what is known in the literature as the weak maximum principle at infinity; \cite{PRS-PAMS}, \cite{PRS-Memoirs}. Indeed, taking $\sigma=\mu=0$,  we deduce that, for a smooth function $u$ on $M$ satisfying $\sup_Mu=u^*<+\infty$, there exists a sequence $\{x_n\}$ along which
$$u(x_n)>u^*-\frac{1}{n}, \text{ and }\Delta_fu(x_n)<\frac{1}{n}.$$

 In general, under volume growth conditions, nothing can be said about the behavior of the gradient of $u$. If,
 along the same sequence $\{x_n\}$, we have  that
 $$|\nabla u(x_n)|<\frac{1}{n}$$
 then we say that the full Omori-Yau maximum principle for the $f$-Laplacian holds.
The following result, which is a generalization of Theorem 1.9 in \cite{PRS-Memoirs}, gives function-theoretic sufficient conditions for a weighted Riemannian manifold $\left(M, \left\langle ,\right\rangle, e^{-f}d\rm{vol}\right)$ to satisfy the Omori-Yau maximum principle.

\begin{theorem}\label{th_OY_fLap}
Let $\left(M, \left\langle \,,\right\rangle, e^{-f}d\rm{vol}\right)$ be a weighted Riemannian manifold and assume that there exists a non-negative $C^{2}$ function $\gamma$ satisfying the following conditions
\begin{eqnarray}
\label{hp1} &\gamma\left(x\right)\rightarrow +\infty \textrm{\,\,\,as\,\,\,} x\rightarrow\infty\\
\label{hp2} &\exists A>0 \textrm{\,\,\,such that\,\,\,} \left|\nabla \gamma\right|\leq A\gamma^{\frac{1}{2}}\textrm{\,\,\, off a compact set}\\
\label{hp3}&\exists B>0 \textrm{\,\,\,such that\,\,\,} \Delta_{f}\gamma\leq B\gamma^{\frac{1}{2}}G\left(\gamma^{\frac{1}{2}}\right)^{\frac{1}{2}} \textrm{\,\,\,off a compact set}\\
\label{hp4} &\exists C>0 \textrm{\,\,\,such that\,\,\,} \left|\nabla f\right|\leq CG\left(\gamma^{\frac{1}{2}}\right)^{\frac{1}{2}}
\end{eqnarray}
where $G$ is a smooth function on $\left[0,+\infty\right)$ satisfying
\begin{equation}\label{hp5}
\begin{array}{lll}
&\left(i\right)\,G\left(0\right)>0&\left(ii\right)\,G^{\prime}\left(t\right)\geq 0 \textrm{\,\,on\,\,} \left[0,+\infty\right)\\
&\left(iii\right)G\left(t\right)^{-\frac{1}{2}}\notin L^{1}\left(+\infty\right)&\left(iv\right)\, \limsup_{t\rightarrow+\infty}\frac{tG\left(t^{\frac{1}{2}}\right)}{G\left(t\right)}<+\infty.
\end{array}
\end{equation}
Then, given any function $u\in C^{2}\left(M\right)$ with $u^{*}=\sup_{M}u<+\infty$, there exists a sequence $\left\{x_{n}\right\}_{n}\subset M$ such that
\begin{equation}\label{O-YMaxPr}
\begin{array}{llll}
&\left(i\right)\,u\left(x_{k}\right)>u^{*}-\frac{1}{k};&\left(ii\right)\,\left|\nabla u\left(x_{k}\right)\right|<\frac{1}{k};&\left(iii\right)\,\Delta_{f}u\left(x_{k}\right)<\frac{1}{k};
\end{array}
\end{equation}
for each $k\in\mathbb{N}$, i.e. the Omori-Yau maximum principle for $\Delta_{f}$
holds on $\left(M,\left\langle ,\right\rangle,
e^{-f}d\mathrm{vol}\right)$.
\end{theorem}
The proof of this theorem is similar to that of Theorem 1.9 in \cite{PRS-Memoirs}
and we refer to this one for more details.
\begin{proof}
We define the function
\begin{equation*}
\varphi\left(t\right)=e^{\int_{0}^{t} G\left(s\right)^{-\frac{1}{2}}ds}.
\end{equation*}
Proceeding as in \cite{PRS-Memoirs} and using assumption (\ref{hp5}) (iv), we have that
\begin{equation}\label{OY1}
0\leq\frac{\varphi^{\prime}\left(t\right)}{\varphi\left(t\right)}
<c\left(tG\left(t^{\frac{1}{2}}\right)\right)^{-\frac{1}{2}}
\end{equation}
for some constant $c>0$. Next, we fix a point $p\in M$ and, $\forall k\in \mathbb{N}$,
 we define
\begin{equation*}
F_{k}(x)=\frac{u\left(x\right)-u\left(p\right)+1}
{\varphi\left(\gamma\left(x\right)\right)^{\frac{1}{k}}}.
\end{equation*}
Then $F_{k}\left(p\right)=1/\varphi\left(\gamma\left(p\right)\right)^{1/k}>0$. Moreover,
since $u^{*}<+\infty$ and $\varphi\left(\gamma\left(x\right)\right)\rightarrow+\infty$
as $x\rightarrow+\infty$, we have $\limsup_{x\rightarrow\infty}F_{k}\left(x\right)\leq0$.
Thus, $F_{k}$ attains a positive absolute maximum at $x_{k}\in M$.
Iterating this procedure, we produce a sequence $\left\{x_{k}\right\}$.
It is shown in \cite{PRS-Memoirs} that
\begin{equation*}
\limsup_{t\rightarrow+\infty}u\left(x_{k}\right)=u^{*},
\end{equation*}
and by passing to a subsequence if necessary, we may assume that
\begin{equation*}
\lim_{k\rightarrow+\infty}u\left(x_{k}\right)=u^{*}.
\end{equation*}
If $\left\{x_{k}\right\}$ remains in a compact set, then $x_{k}\rightarrow\bar{x}$
as $k\rightarrow+\infty$ and the sequence $z_{k}=\bar{x}$, for each $k$, clearly
satisfies (\ref{O-YMaxPr}).
We only need to consider the case when $x_{k}\rightarrow\infty$ so that, according
to (\ref{hp1}), $\gamma\left(x_{k}\right)\rightarrow+\infty$. Since $f_{k}$ attains
a positive maximum at $x_{k}$ we have
\begin{equation*}
\begin{array}{lll}
\left(i\right)\,\left(\nabla \log F_{k}\right)\left(x_{k}\right)=0;
&\left(ii\right)\Delta_{f}
\left(\log F_{k}\right)\left(x_{k}\right)
=\Delta \left(\log F_{k}\right)\left(x_{k}\right)\leq 0.
\end{array}
\end{equation*}
Reasoning as in \cite{PRS-Memoirs} we have
\begin{align*}
\Delta u \left(x_{k}\right) \leq \frac{u\left(x_{k}\right)-u\left(p\right)+1}{k}
&\left\{\frac{\varphi^{\prime}\left(\gamma\left(x_{k}\right)\right)}
{\varphi\left(\gamma\left(x_{k}\right)\right)}
\Delta\left(\gamma\right)\left(x_{k}\right)\right.\\
 \nonumber&\left.+\frac{1}{k}\left(\frac{\varphi^{\prime}
 \left(\gamma\left(x_{k}\right)\right)}
 {\varphi\left(\gamma\left(x_{k}\right)\right)}\right)^{2}
 \left|\nabla\gamma\left(x_{k}\right)\right|^{2}\right\}.
\end{align*}
Assume now that (\ref{hp2}) and (\ref{hp3}) hold so that they
hold at $x_{k}$ for sufficiently large $k$. A computation shows that
\begin{equation*}
\left|\nabla u\left(x_{k}\right)\right|\leq \frac{a}{k}
\cdot\frac{u\left(x_{k}\right)-u\left(p\right)+1}
{G\left(\gamma\left(x_{k}\right)^{1/2}\right)^{1/2}}
\end{equation*}
for some costant $a>0$.
Note that from (\ref{hp2}), (\ref{hp3}), (\ref{hp4}) we have
\begin{align}\label{OY7}
\Delta \gamma\left(x_{k}\right)=&\Delta_{f}\gamma\left(x_{k}\right)
+\left\langle \nabla\gamma\left(x_{k}\right),\nabla f\right\rangle\\
\nonumber\leq&\Delta_{f}\gamma\left(x_{k}\right)
+\left|\nabla\gamma\left(x_{k}\right)\right|\left|\nabla f\right|\\
\nonumber\leq&B\gamma^{1/2}G\left(\gamma^{1/2}\right)^{1/2}
+ AC \gamma^{1/2}G\left(\gamma^{1/2}\right)^{1/2}.
\end{align}
Thus, using (\ref{hp4}), (\ref{OY1}) and (\ref{OY7}) we obtain
\begin{align*}
\Delta_{f}u\left(x_{k}\right)=&\Delta u\left(x_{k}\right)
-\left\langle \nabla u,\nabla f\right\rangle\left(x_{k}\right)\\
\leq&\frac{u\left(x_{k}\right)-u\left(p\right)+1}{k}
\left\{\frac{c}{\gamma^{1/2}G\left(\gamma^{1/2}\right)^{1/2}}D\gamma^{1/2}
G\left(\gamma^{1/2}\right)^{1/2}\right.\\
&\left.+\frac{1}{k}\cdot\frac{c^{2}}{\gamma G\left(\gamma^{1/2}\right)}A^{2}
\gamma\right\}+\frac{a}{k}\cdot\frac{u\left(x_{k}\right)
-u\left(p\right)+1}{G\left(\gamma^{1/2}\right)^{1/2}} CG\left(\gamma^{1/2}\right)^{1/2},
\end{align*}
and the RHS tends to zero as $k\rightarrow+\infty$.
\end{proof}
We now consider two important situations where  Theorem \ref{th_OY_fLap} applies.
First, we recover a result by M. Fern\'andez-L\'opez and E. Garc\'ia R\'io, \cite{LR-Preprint}.
\begin{corollary}
\label{corollary_Lopez}
Let $\left(M,\left\langle \,,\right\rangle,\nabla f\right)$ be a gradient shrinking
Ricci soliton. Then, the Omori-Yau maximum principle for the $f$-Laplacian holds.
\end{corollary}
\begin{proof}
Set $G(t)=t^{2}$ and $\gamma=f$. It is proved in  \cite{LR-Preprint} that
\begin{eqnarray}
     &\gamma\rightarrow +\infty \textrm{\,\,\,as\,\,\,} x\rightarrow\infty\nonumber\\
     &\left|\nabla f\right|=\left|\nabla \gamma\right|\leq\sqrt{f}=\sqrt{\gamma}\nonumber\\
     &\Delta_{f}\gamma =\Delta f -\left|\nabla f\right|^{2}
\leq\Delta f\leq \gamma^{1/2}G\left(\gamma^{1/2}\right)^{1/2},\nonumber
\end{eqnarray}
that is, conditions (\ref{hp1}), (\ref{hp2}), (\ref{hp3}) and (\ref{hp4}) of
Theorem \ref{th_OY_fLap} are satisfied.
\end{proof}
The second situation concerns with general weighted manifolds whose Bakry-Emery Ricci
tensor is suitably controlled. Clearly it represents an  extension of
Corollary~\ref{corollary_Lopez} to non necessarily shrinking almost
solitons.
\begin{corollary}
Let $\left(M, \left\langle ,\right\rangle,e^{-f}d\mathrm{vol}\right)$ be a
complete weighted manifold such that
    \begin{equation*}
    Ric_{f}\geq-\left(m-1\right)G\left(r\right)\left\langle ,\right\rangle
    \end{equation*}
    for a smooth positive function $G$ satisfyng (\ref{hp5}), even at the origin.
    Assume also that
    \begin{equation*}
    \left|\nabla f\right|\leq CG\left(r\right)^{1/2}
    \end{equation*}
Then, the Omori-Yau maximum principle for the $f$-Laplacian holds on $M$.
\end{corollary}
\begin{proof}
    Let $h$ be as in Theorem \ref{3.1}. Then
    \begin{align*}
    \Delta_{f}r\leq&(m-1)\frac{h^{\prime}}{h}+\left|\nabla r\right|\left|\nabla f\right|\\
    \leq&\left(m-1\right)\frac{h^{\prime}}{h}+CG\left(r\right)^{1/2}\leq DG\left(r\right)^{1/2},
    \end{align*}
    and thus
    \begin{align*}
    \Delta_{f}r^{2}=&2+2r\Delta_{f}r\\
    \leq&2+2rG\left(r\right)^{1/2}\\
    \leq&CrG\left(r\right)^{1/2},
    \end{align*}
    and the hypotheses of Theorem \ref{th_OY_fLap} are satisfied with $\gamma=r^{2}$.
\end{proof}

\section{Triviality in the presence of weighted Poincar\'e-Sobolev inequalities}
This section aims to prove Theorem \ref{C} of the Introduction. In a sense it
can be considered as an isolation result for the soliton function of almost
solitons with $L^p$ soliton structure. Again we need a preliminary result.
The next proposition can be deduced by simple modifications to the proof of
Theorem 9.12 in \cite{PRS-Progress}.
\begin{proposition}\label{prop_4.1}
Let $\left(M,\left\langle\, ,\right\rangle, e^{-f}d\mathrm{vol}\right)$ be
a complete weighted manifold and assume that, for some $0\leq\alpha<1$, the
Poincar\'e-Sobolev inequality
\begin{equation}\label{4.1}
\int_{M}\left|\nabla \varphi\right|^{2}e^{-f}d\mathrm{vol}
\geq S\left(\alpha\right)^{-1}\left\{\int_{M}\left|\varphi\right|^{\frac{2}{1-\alpha}}
e^{-f}d\mathrm{vol}
\right\}^{1-\alpha},
\end{equation}
holds for all $\varphi\in C^{\infty}_{0}\left(M\right)$ and some
constant $S\left(\alpha\right)>0$. Let $B\in\mathbb{R}$,
$q\left(x\right)\in C^{0}\left(M\right)$ and let $\psi\in
Lip_{loc}\left(M\right)$ be a non-negative weak solution of
\begin{equation*}
\psi\Delta_{f}\psi+q\left(x\right)\psi^{2}\geq -B\left|\nabla \psi\right|^{2}
\end{equation*}
on $M$. Assume that, for some
\begin{equation*}
\sigma>\max\left\{1,B+1\right\}
\end{equation*}
we have
\begin{equation*}
\int_{B_{r}}\psi^{\sigma}e^{-f}=o\left(r^{2}\right)
\end{equation*}
as $r\rightarrow+\infty$. Then either $\psi\equiv0$ or otherwise
\begin{equation*}
\left\|q_{+}\left(x\right)\right\|_{L^{\frac{1}{\alpha}}\left(M,e^{-f}\right)}
\geq\frac{4}{S\left(\alpha\right)}\frac{p-1}{p^{2}}.
\end{equation*}
\end{proposition}
\begin{remark}
\rm{
As we already observed in the Introduction, the validity of a Poincar\'e-Sobolev
inequality forces $M$ to have infinite weighted volume. Furthermore, if $\alpha=0$,
(\ref{4.1}) means precisely that the manifold has positive weighted spectral radius.
}
\end{remark}
An immediate consequence of Proposition \ref{prop_4.1} is the following
\begin{theorem}\label{th_4.1}
Let $\left(M,\left\langle\, ,\right\rangle,\nabla f\right)$ be a complete, gradient Ricci
almost soliton with soliton function $\lambda$. Suppose either $m=2$ or otherwise
\begin{equation}\label{4.6}
\left\langle \nabla f,\nabla\lambda\right\rangle\leq0.
\end{equation}
Assume the validity of the Poincar\'e-Sobolev inequality (\ref{4.1}) for some
$0\leq\alpha<1$ and suppose that for some $p>1$
\begin{equation*}
\int_{B_{r}}\left|\nabla f\right|^{p}e^{-f}=o\left(r^{2}\right)
\end{equation*}
as $r\rightarrow+\infty$. Then either the almost soliton is trivial or
\begin{equation}\label{4.8}
\left\|\lambda_{+}\left(x\right)\right\|_{L^{\frac{1}{\alpha}}\left(M,e^{-f}\right)}
\geq\frac{4}{S\left(\alpha\right)}\frac{\sigma-1-B}{\sigma^{2}}.
\end{equation}
\end{theorem}
\begin{proof}
From Corollary \ref{cor_2.11} and (\ref{4.6}) we deduce
\begin{equation*}
\left|\nabla f\right|\Delta_{f}\left|\nabla f\right|+\lambda\left|\nabla f\right|^{2}\geq0.
\end{equation*}
Thus we can apply Proposition \ref{prop_4.1} with $p=\sigma$,
$B=0$, $q\left(x\right)=\lambda\left(x\right)$ to deduce that either the almost
soliton is trivial
 or (\ref{4.8}) holds.
\end{proof}
Theorem \ref{C} now follows immediately from Theorem \ref{th_4.1}.
Indeed, since $m\geq 3$, a trivial almost soliton is necessarily Einstein and the soliton
function $\lambda$ must be constant. On the other hand, the Poincar\'e-Sobolev
inequality implies that $M$ has infinite volume, therefore $\lambda_+=0$.

\section{Scalar curvature estimates}
In this section we prove Theorem \ref{B} stated in the Introduction.
Namely, we show that under a pointwise control on the soliton function,
the scalar curvature of an almost soliton is bounded from below. Furthermore,
the lower bound of the scalar curvature can be estimated both from above and
from below and some rigidity at the endpoints occurs.

We recall that a weighted manifold $\left(M,\left\langle ,\right\rangle,e^{-f}d\rm{vol}\right)$
is said to be $f$-parabolic if whenever $w$ is bounded below and satisfies
$\Delta_{f}w\leq0$ then $w$ is constant.

\begin{proof}(of Theorem \ref{B}). Since $|\mathrm{Ric}|^2\geq S^2/m$ by the Cauchy-Schwarz
inequality, and $\Delta \lambda \leq 0$ by assumption, (\ref{2.14}) in Lemma
\ref{lemma_2.2} yields
\begin{equation}
\label{5.1}
\frac 12  \Delta_{f}S  =\lambda S- \left|Ric\right|^{2}+ \left(m-1\right)\Delta\lambda\leq
\lambda S- \frac {S^2}m .
\end{equation}

Note next that since $\mathrm{Ric}_f=\lambda \geq \lambda _* >-\infty
$, by Proposition~\ref{prop_volest1} we have
\begin{equation*}
\mathrm{vol}_f(B_r)\leq C_1 e^{C_2 r^2},
\end{equation*}
for some positive constants $C_1,$ $C_2,$
and, in particular,
\begin{equation}
\label{volume estimate}
\liminf_{r\to\infty} \frac{\log \mathrm{vol}_f(B_r)}{r^2} \leq
C_2<+\infty.
\end{equation}
Applying  Theorem 12 of \cite{PRiS} to the function $S_-=\max\{-S, 0\}$, which
 is a weak solution of
\begin{equation*}
 \Delta_{f}S_- \geq  2\lambda S_- - \frac{2}{m}S_-^{2},
\end{equation*}
with $a\left(x\right)=2 \lambda\left(x\right)$, $b\left(x\right)=\frac{2}{m}$,
$\sigma=2$, and  deduce
that
\begin{equation*}
S_-(x)\leq \sup_M \frac{ \lambda_-(x)}{1/m} ,
\end{equation*}
from which we conclude that
\begin{equation*}
S(x)\geq \min\{m\lambda_*, 0\}.
\end{equation*}
In particular, $S_*\geq 0$ if $\lambda\geq 0$, and $S_*\geq m\lambda_*$ if
$\lambda_*\leq \lambda\leq 0$.

Next, again by (\ref{volume estimate}), the weak minimum principle
at infinity for $\Delta_f$ holds (see Theorem~\ref{prop_3.1}), and therefore we
may find a sequence $\{x_n\}$ such that $\Delta_f S(x_n)\geq - 1/n$ and $S(x_n)\to S_*.$
Computing the liminf of (\ref{5.1}) along this sequence and setting
$\overline \lambda = \liminf \lambda(x_n)$ we deduce
that
\begin{equation*}
0\leq  \overline \lambda S_*-S_*^2/m.
\end{equation*}
Thus, if $\overline \lambda = 0$, then $S_*=0,$ while if $\overline \lambda \ne
0,$ then solving the inequality yields
$m\overline \lambda \leq S_* \leq 0$ if $\overline \lambda <0$ and
$0\leq S_* \leq m\overline\lambda $ if $\overline \lambda >0$. Since obviously
$\lambda_*\leq \overline \lambda \leq \lambda^*$, this gives the
scalar curvature estimates in (i), (ii) and (iii).

We now suppose that the scalar curvature achieves its lower bound and,
according to the classification in Theorem \ref{th_classification}, we prove rigidity.

In case (i) using (\ref{5.1}), we see that $S(x)\geq S_*=m\lambda_*$ satisfies
\begin{equation}
\label{chain2}
\frac{1}{2}\Delta_{f}S \leq -\frac{S}{m}(S-m\lambda)\leq -\frac{S}{m}(S-m\lambda_*)
\end{equation}
on the open set $\Omega=\{x\in M:S(x)<0\}$. Therefore, if $S(x_0)=S_*=m\lambda_*<0$
for some $x_0$, we deduce that the function $u=S-m\lambda_*\geq 0$ achieves its
minimum value $u(x_0)=0$ and satisfies the differential inequality
\[
\frac{1}{2}\Delta_{f}u+\lambda_* u\leq0
\]
on $\Omega$. By the minimum principle, $u(x)=0$ on the connected
component $\Omega_0$ of $\Omega$ containing $x_0$. It follows
that the open set $\Omega_0$ is also closed, thus $\Omega_0=M$ and
$u(x)=0$ on $M$. This means that $S(x)=m\lambda_*$ is constant.
Using this information into (\ref{chain2}) we get that $\lambda$ is constant.
Going back to  (\ref{5.1}), by the equality case in the Cauchy-Schwarz inequality, we obtain
\[
\ Ric=\lambda\left\langle ,\right\rangle,
\]
showing that $M$ is Einstein and the soliton is trivial.

Assume next that the soliton is steady, so that $\lambda \equiv 0$.
Then  $S(x)\geq S_*=0$ solves
\[\frac{1}{2}\Delta_{f}S\leq0.\]
Therefore, if $S(x_0)=0$, arguing as above we conclude that $M$ must be Ricci-flat and,
by case (a.1) of Theorem  \ref{th_classification}, $M$ is a cylinder over a Ricci-flat,
totally geodesic hypersurfaces $\Sigma$.

Finally assume that the almost soliton is shrinking. Then, $S(x)\geq S_*=0$ satisfies
\[
\frac{1}{2}\Delta_{f}S\leq \lambda S.
\]
If $S(x_0)=0$ for some $x_0 \in M$, by the minimum principle (see e.g. page 35 in \cite{GT})
we deduce that $S(x)=0$ is constant and all the inequalities used in (\ref{5.1}) become
equalities. In particular,  $|Ric|^2=S^2/m=0$ proving that $M$ is Ricci-flat.
By case (a.2) in Theorem \ref{th_classification}, $\lambda$ is a positive constant and $M$
must be isometric to the standard Euclidean space.

It remains to prove the last statement. Suppose then that $S_{*}=m\lambda^{*}>0$.
Since $S\geq S_{*}=m\lambda^{*}\geq m\lambda>0$, it follows that $m\lambda S-S^2\leq 0$
on $M$. Thus from (\ref{5.1}),
\[
\ \Delta_{f}S\leq 0
\]
and $S>0$ is a nonnegative $\Delta_f$-superharmonic function. It follows that, if
$\left(M,\left\langle ,\right\rangle, e^{-f}d\mathrm{vol}\right)$ is $f$-parabolic, then
$S=m\lambda^*$ is constant. Using (\ref{5.1})
we immediately deduce $S\left(\lambda-\frac{S}{m}\right)=0$ so that
$\lambda=\frac{S}{m}$ is constant. From (\ref{5.1}) we  have that
$\left|Ric\right|^{2}=\frac{S^2}{m}$, and, again by the equality case
 in the Cauchy-Schwarz inequality
\[
\ Ric=\lambda\left\langle\, ,\right\rangle,
\]
with $\lambda>0$. Thus $M$ is Einstein and compact by Myers' Theorem. Now from (\ref{0.2})
and the above considerations it follows that $Hess\left(f\right)=0$ on $M$
and compactness implies that $f$ is constant. Finally, if $\lambda \geq
A^2(1 +r)^{-\mu}$ with $0\leq \mu<1$, then by Proposition~\ref{prop_volest1}
 $\mathrm{vol}_f (\partial B_r)$ tends to zero as $r\to \infty,$ so, in particular,
\[
\frac 1 {\mathrm{vol}_f (\partial
B_r)} \not\in L^1(+\infty)
\]
and adapting to the diffusion
operator $\Delta_f$ standard proofs valid for the ordinary Laplace-Beltrami operator, \cite{RS},
one shows that $\left(M,\left\langle\, ,\right\rangle, e^{-f}d\mathrm{vol}\right)$,
is $f$-parabolic.
\end{proof}

\section{A gap theorem for the traceless Ricci tensor}

\begin{proof}(of Theorem \ref{D}).
As noted in the previous proof, since  $\mathrm{Ric}_f$ is bounded below by
(\ref{0.21}),  the weak maximum principle for $\Delta_{f}$ holds on
$\left(M,\left\langle\, ,\right\rangle,\nabla f\right)$.
Next, by Corollary~\ref{2.24}, (\ref{0.18}), (\ref{0.19}), and (\ref{0.21}),
we deduce that
\[
\ \frac{1}{2}\Delta_{f}\left|T\right|^{2}\geq 2\left(\lambda_{*}-
S^{*}\frac{m-2}{m\left(m-1\right)}\right)\left|T\right|^{2}-
\frac{4}{\sqrt{m\left(m-1\right)}}\left|T\right|^{3}.
\]
Assuming that $\left|T^{*}\right|=\sup_{M}\left|T\right|<+\infty$
(for otherwise there is nothing to prove) we may apply the weak maximum
principle and deduce that either
$\left|T\right|^{*}=0$ or
$\left|T\right|^{*}\geq\frac{1}{2}\left(\sqrt{m\left(m-1\right)}
\lambda_{*}-S^{*}\frac{m-2}{\sqrt{m\left(m-1\right)}}\right)$. In
the former case, $T=0$ that is $\mathrm{Ric}= S/m \langle \,,\rangle
$ and since $m\geq 3$, $S$ is constant by Schur's lemma and  $M$ is
Einstein, as required to conclude the proof of Theorem~\ref{D}.
\end{proof}

\section{Some topological remarks}
We shall prove the next theorem, which extends results obtained
in \cite{W}, \cite{N}, \cite{LR-Annalen} and \cite{ELNM}.
\begin{theorem}\label{teotopolog}
Let $\left(M,\left\langle ,\right\rangle\, ,e^{-f}d\mathrm{vol}\right)$ be a
geodesically complete weighted manifold, and assume that there
exists a point $o\in M$  and functions $\mu\geq 0$ and $g$ bounded
such that for every unit speed geodesic $\gamma$
issuing from $\gamma (0)=o$ we have
\begin{equation*}
 \mathrm{Ric}_f (\dot\gamma,\dot \gamma) \geq \mu\circ\gamma +
 \langle \nabla g\circ \gamma, \dot\gamma\rangle
\end{equation*}
and
\begin{equation*}
\int_0^{+\infty} \mu\circ \gamma (t) dt =+\infty.
\end{equation*}
Then, the following hold:
\begin{enumerate}
\item[(a)]If the above conditions hold then
 $\left|\pi_{1}\left(M\right)\right|<\infty$.
\item[(b)]If in addition $Ric\leq c<+\infty$ and $\mu=\mu_o(r(x))$ is
radial, where $r\left(x\right)=\mathrm{dist}\left(x,o\right)$,
then $M$ is diffeomorphic to the interior of  a compact manifold $N$ with $\partial N\neq \emptyset$.
\item[(c)]If $\mu\left(x\right)\geq\mu_{0}>0$
and $\sup_{M}( \left|\nabla f\right| +  |g|) \leq F<+\infty$,
then $M$ is compact and
$diam(M)\leq\frac{1}{\mu_{0}}\left[2F+\sqrt{4F^{2}+\pi^{2}\left(m-1\right)c}\right]$
\end{enumerate}
\end{theorem}
Clearly the Theorem applies to almost Ricci solitons for which the
soliton function $\lambda$ satisfies the conditions listed in the
statement.

These three conclusions can be deduced from the following lemmas which
estimates the integral of $\mathrm{Ric}$ along geodesics.
\begin{lemma}
\label{lemmatopolog1}
Let $\left(M,\left\langle ,\right\rangle\right)$ be a Riemannian manifold.
Fix $o\in M$ and let $r\left(x\right)=dist\left(x,o\right)$. For any point
$q \in M$, let $\gamma_{q}:\left[0,r\left(q\right)\right]\rightarrow M$ be
a minimizing geodesic from $o$ to $q$ such that $\left|\dot{\gamma}_{q}\right|=1$.
\begin{itemize}
\item[(A)]
If $h\in Lip_{loc}\left(\mathbb{R}\right)$ is such that $h\geq 0$ and $h\left(0\right)=0$,
then, for every $q\not\in \mathrm{cut}(o)$,
\begin{equation*}
h^{2}\left(r\left(q\right)\right)\Delta r (q)
\leq
\left(m-1\right)\int^{r\left(q\right)}_{0}\left(h^{\prime}\right)^{2}ds
-\int^{r\left(q\right)}_{0}h^{2}Ric\left(\dot{\gamma_{q}},\dot{\gamma_{q}}\right)ds
\nonumber.
\end{equation*}
If in addition  $h\left(r\left(q\right)\right)=0$,
then for every $q \in M $,
\[
\ 0\leq \left(m-1\right)\int^{r\left(q\right)}_{0}\left(h^{\prime}\right)^{2}ds
-\int^{r\left(q\right)}_{0}h^{2}Ric\left(\dot{\gamma_{q}},\dot{\gamma_{q}}\right)ds.
\]
\item[(B)] For all $q\in M$ such that $r\left(q\right)>2$, we have
\[
\ \int^{r\left(q\right)}_{0}Ric\left(\dot{\gamma_{q}},\dot{\gamma_{q}}\right)
\leq2\left(m-1\right)+H_{o}+H_{q},
\]
where, as in \cite{W}, we have set
\[
\ H_{p}=\max\left\{0,\sup_{B_{1}\left(p\right)}\mathrm{Ric}\right\},\,\forall \,p\in M
\]
\end{itemize}
\end{lemma}
\begin{proof}
Part (A) is well known. We provide a proof which avoids  the use of the second variation
formula for arc-length.

Fix a point $q\in M$ and suppose that $q\notin cut\left(o\right)$.
Proceeding as in the proof of Theorem \ref{3.1} we rewrite (\ref{3.8}) as
\begin{equation}
\frac{\left(\Delta r\circ\gamma_{q}\right)^{2}}{m-1}+\frac{d}{ds}
\left(\Delta r\circ\gamma_{q}\right)
+Ric\left(\dot{\gamma_{q}},\dot{\gamma_{q}}\right)\leq 0.
\label{topolog3}
\end{equation}
If $h\in Lip_{loc}\left(\mathbb{R}\right)$ is
such that $h\geq 0$, $h\left(0\right)=0$, multiplying by $h^{2}$
equation (\ref{topolog3}) and integrating on $\left[0,t\right]$ we
get
\[
\int^{t}_{0}h^{2}\frac{\left(\Delta r\circ\gamma_{q}\right)^{2}}{m-1}ds
+\int^{t}_{0}\frac{d}{ds}\left(\Delta r\circ\gamma_{q}\right)h^{2}ds
+\int^{t}_{0}h^{2}Ric\left(\dot{\gamma_{q}},\dot{\gamma_{q}}\right)ds\leq 0.
\]
Integrating by parts and noting that
$\left(\Delta r \circ \gamma_{q}\right)h^{2}\rightarrow 0$
as $r\rightarrow 0$ we obtain
\begin{align*}
\ 0\geq &\int^{t}_{0}h^{2}\frac{\left(\Delta r \circ \gamma_{q}\right)^{2}}{m-1}ds
+h^{2}\left(t\right)\Delta r \circ \gamma_{q}\left(t\right)\\
&-2\int^{t}_{0}hh^{\prime}\left(\Delta r \circ \gamma_{q}\right)ds
+\int^{t}_{0}Ric\left(\dot{\gamma_{q}},\dot{\gamma_{q}}\right)h^{2}ds,
\end{align*}
Since
\begin{align*}
2hh^{\prime}\left(\Delta r\circ\gamma_{q}\right)
=&\frac{2h\Delta r\circ\gamma_{q}}{\sqrt{m-1}}\sqrt{m-1}h^{\prime}\\
\leq&\frac{h^{2}\left(\Delta r\circ\gamma_{q}\right)^{2}}{m-1}
+\left(m-1\right)\left(h^{\prime}\right)^{2},
\end{align*}
we deduce that
\[
0\geq h^{2}\left(t\right)\Delta r\circ\gamma_{q}\left(t\right)
-\left(m-1\right)\int^{t}_{0}\left(h^{\prime}\right)^{2}ds
+\int^{t}_{0}h^{2}Ric\left(\dot{\gamma_{q}},\dot{\gamma_{q}}\right)ds
\]
and, setting  $t=r\left(q\right)$, we conclude that
\begin{equation*}
h^{2}\left(r\left(q\right)\right)\Delta r (q)
\leq
\left(m-1\right)\int^{r\left(q\right)}_{0}\left(h^{\prime}\right)^{2}ds
-\int^{r\left(q\right)}_{0}h^{2}Ric\left(\dot{\gamma_{q}},\dot{\gamma_{q}}\right)ds\nonumber
\end{equation*}
If  in addition $h$ satisfies  $h^{2}\left(r\left(q\right)\right)=0$, then the above inequality
becomes
\begin{equation}\label{topolog4a}
0\leq \left(m-1\right)\int^{r\left(q\right)}_{0}\left(h^{\prime}\right)^{2}ds
-\int^{r\left(q\right)}_{0}h^{2}Ric\left(\dot{\gamma_{q}},\dot{\gamma_{q}}\right)ds.
\end{equation}
This completes the proof if $q\notin cut\left(o\right)$. In the general case
inequality (\ref{topolog4a}) can be extended to any
$q \in M$ using the Calabi trick. Indeed, suppose that $q\in cut\left(o\right)$.
Translating the origin $o$ to $o_{\epsilon}=\gamma_{q}\left(\epsilon\right)$
so that $q\notin cut\left(o_{\epsilon}\right)$, using the triangle inequality
and, finally, taking the limit as $\epsilon\rightarrow 0$, one checks that
(\ref{topolog4a}) holds also in this case.

To prove part (B) we note that if  $h\in Lip_{loc}\left(\mathbb{R}\right)$
is such that $h\geq 0$ and $h\left(0\right)=h\left(r\left(q\right)\right)=0$,
then by we may rewrite (A) in the form
\begin{equation*}
\int^{r\left(q\right)}_{0}\!Ric\left(\dot{\gamma_{q}},\dot{\gamma_{q}}\right)ds
\leq
\left(m-1\right)\int^{r\left(q\right)}_{0}\!\left(h^{\prime}\right)^{2}ds
+\int^{r\left(q\right)}_{0}\!\left(1-h^{2}\right)
Ric\left(\dot{\gamma_{q}},\dot{\gamma_{q}}\right)ds.\nonumber
\end{equation*}
Choosing
\[
\ h\left(s\right)=\left\{
\begin{array}{ll}
s&0\leq s\leq1\\
1&1\leq s\leq r\left(q\right)-1\\
r\left(q\right)-s&r\left(q\right)-1\leq s\leq r\left(q\right),
\end{array}
\right.
\]
where $r\left(q\right)>2$, we obtain
\begin{align*}
\int^{r\left(q\right)}_{0}Ric\left(\dot{\gamma_{q}},\dot{\gamma_{q}}\right)ds\leq&2\left(m-1\right)+\int^{1}_{0}\left(1-s^{2}\right)Ric\left(\dot{\gamma_{q}},\dot{\gamma_{q}}\right)ds\\
+&\int^{r\left(q\right)}_{r\left(q\right)-1}\left(1-\left(r\left(q\right)-s\right)^{2}\right)Ric\left(\dot{\gamma_{q}},\dot{\gamma_{q}}\right)ds\\
\leq&2\left(m-1\right)+H_{o}+H_{q}.
\end{align*}

\end{proof}

\begin{lemma}\label{lemmatopolog}
Let $\left(M,\left\langle ,\right\rangle,e^{-f}d\mathrm{vol} \right)$ be a
complete weighted Riemannian manifold. Fix $o\in M$ and let
$r\left(x\right)=dist\left(x,o\right)$ and assume that there exist functions
$\mu\geq 0$ and $g$ bounded such that for every unit speed
geodesic $\gamma$ issuing from $o$
\begin{equation*}
\mathrm{Ric}_f (\dot\gamma,\dot\gamma) \geq \mu (\gamma(t)) +
\langle \nabla g,\dot\gamma\rangle.
\end{equation*}
Then for every such geodesic
\begin{align*}
\int^{t}_{0}Ric\left(\dot{\gamma},\dot{\gamma}\right)
=&\left\langle \nabla f,\dot{\gamma}\left(0\right)\right\rangle
- \left\langle \nabla f, \dot{\gamma}\left(t\right)\right\rangle
+\int_{0}^{t} \mu\left(\gamma\left(s\right)\right) ds + g(\gamma(t)) - g(o)\\
& \geq-\left|\nabla f_{\gamma\left(0\right)}\right|-
\left|\nabla f_{\gamma\left(t\right)}\right|- 2\sup |g| +\int_{0}^{t}
\mu\left(\gamma\left(s\right)\right) ds .
\end{align*}
\end{lemma}

\begin{proof}
By assumption
\begin{equation}
\label{ric_est}
\ Ric\left(\dot{\gamma},\dot{\gamma}\right)
+Hess\left(f\right)\left(\dot{\gamma},\dot{\gamma}\right)
\geq \mu\circ\gamma + \langle \nabla g, \dot\gamma\rangle,
\end{equation}
which can be written in the form
\[
\ Ric\left(\dot{\gamma},\dot{\gamma}\right)
+\frac{d}{dt}\left\langle \nabla f\left(\gamma\right),\dot{\gamma}\right\rangle
\geq \mu\circ\gamma + \frac {d}{dt} (g\circ \gamma).
\]
Now integrating on $\left[0,t\right]$,
\begin{equation*}
\int^{t}_{0}Ric\left(\dot{\gamma},\dot{\gamma}\right)
+\left\langle \nabla f,\dot{\gamma}(t) \right\rangle
-\left\langle \nabla f,\dot{\gamma}(0)\right\rangle
\geq \int^{t}_{0}\mu\left(\gamma\left(s\right)\right)ds + g(\gamma (t)) - g(o).
\end{equation*}
\end{proof}

We are now in the position to give the
\begin{proof}[Proof of Theorem \ref{teotopolog}]
Following \cite{W},  let us consider the
Riemannian universal covering $P: M^{\prime}\rightarrow M$ \ of $M$. Since $P$
is a local isometry then $M^{\prime}$ is a weighted complete
Riemannian manifold with weight $ f^{\prime} = f\circ P$. Moreover,
since every unit speed geodesic $\gamma^{\prime}$ projects to a unit
speed geodesic $\gamma = P \circ\gamma^{\prime}$ we see that
\begin{equation*}
\mathrm{Ric}_{f^{\prime}}^{\prime}(\dot \gamma^{\prime},\dot\gamma^{\prime}) =
\mathrm{Ric}_f (\dot\gamma,\dot\gamma) \geq \mu \circ\gamma +
\frac {d}{dt} (g\circ \gamma) = \mu^{\prime}\circ\gamma^{\prime}
+\frac{d}{dt} (g^{\prime}\circ \gamma^{\prime}),
\end{equation*}
where the function $g^{\prime}= g\circ P$ is bounded and

$\mu ^{\prime} = \mu\circ P \geq 0$.
satisfies
\begin{equation}
\label{lambdaprime divergence}
\int_0^{+\infty}\mu^{\prime}\circ\gamma^{\prime} (t)dt =
\int_0^{+\infty} \mu\circ\gamma (t)dt = +\infty.
\end{equation}

We identify
\[
\ \pi_{1}\left(  M,o\right)  =Deck(M^{\prime}),
\]
the covering transformation group, and recall that there is a bijective
correspondence $\pi_{1}\left(  M,o\right)  \longleftrightarrow P^{-1}\left(
o\right)  $. Therefore it suffices to show that $P^{-1}\left(  o\right)
\subset B_{R}^{\prime}\left(  o^{\prime}\right)  $ for some $R>>1$. Since
$\pi_{1}\left(  M,o\right)  =Deck(M^{\prime})$ acts transitively on the fibre
$P^{-1}\left(  o\right)  $, we have
\[
\ P^{-1}\left(  o\right)  =\left\{  h\left(  o^{\prime}\right)  :h\in
Deck(M^{\prime})\right\}  ,
\]
and we are reduced to showing that
\[
r^{\prime}\left(  h\left(  o^{\prime}\right)  \right)  \leq R<\infty
,\;\,\,\forall h\in Deck(M^{\prime}),
\]
where we have set $r^{\prime}\left(  x^{\prime}\right)  =dist_{M^{\prime}%
}\left(  o^{\prime},x^{\prime}\right)  $. Fix $h\in Deck(M^{\prime})$ and a
unit speed, minimizing geodesic $\ \gamma_{h\left(  o^{\prime}\right)
}^{\prime}:\left[  0,r^{\prime}\left(  h\left(  o^{\prime}\right)  \right)
\right]  \rightarrow M^{\prime},$ issuing from $\ \gamma_{h\left(  o^{\prime
}\right)  }^{\prime}\left(  0\right)  =o^{\prime}$.
Recalling that $\mathrm{Ric}^{\prime}(\dot\gamma^{\prime},\dot\gamma^{\prime})
= \mathrm{Ric}^{\prime}_{f^{\prime}} (\dot\gamma^{\prime},\dot\gamma^{\prime}) -
\frac{d}{dt} \langle \nabla^{\prime} f^{\prime}\circ \gamma^{\prime},
\dot \gamma^{\prime}\rangle$ and using Lemma~\ref{lemmatopolog1} (B)
and Lemma~\ref{lemmatopolog}
we get
\begin{multline*}
\int_{0}^{r^{\prime}\left(  h\left(  o^{\prime}\right)  \right)  }%
\mu^{\prime}\circ\ \gamma_{h\left(  o^{\prime}\right)  }^{\prime}\left(
s\right)  ds
\leq2\left(  m-1\right)  +H_{o^{\prime}}+H_{h\left(  o^{\prime }\right)  }
\\
+\ \left\vert \nabla^{\prime}f^{\prime}\right\vert \left(
o^{\prime}\right)  +\left\vert \nabla^{\prime}f^{\prime}\right\vert \left(
h\left(  o^{\prime}\right)  \right)  + 2\sup_{M^\prime} |g^\prime|.
\end{multline*}
Since $P:M^{\prime}\rightarrow M$ is a local isometry and $o^{\prime},h\left(
o^{\prime}\right)  \in P^{-1}\left(  o\right)  $ we deduce
\[
\ \left\vert \nabla^{\prime}f^{\prime}\right\vert \left(  o^{\prime}\right)
=\left\vert \nabla f\right\vert \left(  o\right)  =\left\vert \nabla^{\prime
}f^{\prime}\right\vert \left(  h\left(  o^{\prime}\right)  \right)  .
\]
On the other hand $Deck(M^{\prime})\subset Iso(M^{\prime})$, so $h(B_{1}%
^{\prime}(o^{\prime}))$ is isometric to $B_{1}^{\prime}\left(  h\left(
o^{\prime}\right)  \right)  $ and we have
\[
\ \left\vert H_{o^{\prime}}\right\vert =\left\vert H_{h\left(  o^{\prime
}\right)  }\right\vert .
\]
Summarizing, we have obtained that, for every $h\in Deck\left(  M^{\prime
}\right)  $,%
\begin{equation}
\int_{0}^{r^{\prime}\left(  h\left(  o^{\prime}\right)  \right)  }%
\mu^{\prime}\circ\ \gamma_{h\left(  o^{\prime}\right)  }^{\prime}\left(
s\right)  ds\leq2\left\{  \left(  m-1\right)  +H_{o^{\prime}}+\left\vert
\nabla f\right\vert \left(  o\right)  \right\} +2\sup_{M} |g| .
\label{wylie_inequality}%
\end{equation}
With this preparation, we now argue by contradiction and suppose that there
exists a sequence of transformations $\left\{  h_{n}\right\}  \subset
Deck\left(  M^{\prime}\right)  $ \ such that%
\begin{equation}
r^{\prime}\left(  h_{n}\left(  o^{\prime}\right)  \right)  \rightarrow
+\infty\text{, as }n\rightarrow+\infty.\label{wylie_diverging}%
\end{equation}
Let  $\gamma_{h_{n}\left(  o^{\prime}\right)  }^{\prime}\left(  s\right)
=\exp_{o^{\prime}}\left(  s\xi_{n}^{\prime}\right)  $, where $\left\{  \xi
_{n}^{\prime}\right\}  \subset\mathbb{S}_{o^{\prime}}^{m-1}\subset
T_{o^{\prime}}M^{\prime}$. Then, there exists a subsequence $\left\{
\xi_{n_{k}}^{\prime}\right\}  \rightarrow\xi^{\prime}\in\mathbb{S}_{o^{\prime
}}^{m-1}$ as $k\rightarrow+\infty$ and, by the Ascoli-Arzel\`a's Theorem, the
sequence of minimizing geodesics $\left\{  \gamma_{h_{n_{k}}\left(  o^{\prime
}\right)  }^{\prime}\right\}  $ converges uniformly on compact subintervals of
$[0,+\infty)$ to the unit speed geodesic $\gamma^{\prime}\left(  s\right)
=\exp_{o^{\prime}}\left(  s\xi^{\prime}\right)  $.
Since, by (\ref{lambdaprime divergence})
\[
\int_{0}^{+\infty}\mu^\prime\circ\gamma^\prime\left(  s\right)  ds=+\infty,
\]
we can choose $T>>1$ such that%
\begin{equation}
\int_{0}^{T}\mu^\prime\circ\gamma^\prime\left(  s\right)  ds>2\left\{  \left(
m-1\right)  +H_{o^{\prime}}+\left\vert \nabla f\right\vert \left(  o\right)
\right\} +2\sup_{M} |g|  .\label{wylie_inequality2}%
\end{equation}
On the other hand, according to (\ref{wylie_diverging}) we find $k_{0}>0$ such
that, for every $k\geq k_{0},$ $r^{\prime}\left(  h_{n_{k}}\left(  o^{\prime
}\right)  \right)  >T$. It follows from this, from inequality
(\ref{wylie_inequality}) and the definition of $\mu^{\prime}\left(
x^{\prime}\right)  =\mu\circ P\left(  x^{\prime}\right) \geq0$ that%
\begin{align*}
\int_{0}%
^{T}\mu^{\prime}\circ\ \gamma_{h_{n_{k}}\left(  o^{\prime}\right)
}^{\prime}\left(  s\right)  ds &
 \leq\int_{0}^{r^{\prime}\left(  h_{n_{k}}\left(  o^{\prime}\right)  \right)
}\mu^{\prime}\circ\ \gamma_{h_{n_{k}}\left(  o^{\prime}\right)  }^{\prime
}\left(  s\right)  ds\\
& \leq2\left\{  \left(  m-1\right)  +H_{o^{\prime}}+\left\vert \nabla
f\right\vert \left(  o\right)  \right\} +2\sup_{M} |g|  .
\end{align*}
Whence, letting $k\rightarrow+\infty$ we deduce%
\[
\int_{0}^{T}\mu^{\prime}\circ\gamma^{\prime}\left(  s\right)  ds\leq2\left\{  \left(
m-1\right)  +H_{o^{\prime}}+\left\vert \nabla f\right\vert \left(  o\right)
\right\}+2\sup_{M} |g|
\]
which contradicts (\ref{wylie_inequality2}).

Now for the proof of (b), suppose $Ric\leq c$.
Fix $q\in M$ such that $r (q )=\mathrm{dist} (o,q )>2$, and let $\gamma_q$
be a minimizing geodesic
joining $o$ to $q$. As above, combining (B) of Lemma \ref{lemmatopolog1} and
Lemma~\ref{lemmatopolog}, and recalling that $\mu(x)=\mu_o(r(x))$ is radial we obtain
\begin{align*}
\ - |\nabla f (o ) |-
 |\nabla f (\gamma (r(q )) ) | - 2\sup_M|g|
+\int^{r (q )}_{0}\mu_o(s)ds
\leq& 2 (m-1 )+ H_{q}+H_{o}\\ \leq& 2 (m-1 )+2c,
\end{align*}
which implies
\[
\  |\nabla f (q ) |\geq\int^{r (q )}_{0}\mu_o (s )ds
+ \{- |\nabla f (o ) |-2 (m-1 )-2c \}.
\]
Since $0<\mu_o\notin L^{1} (+\infty )$ if $r (q )>>1$, say
$r (q )\geq R_{0}$, we have $ |\nabla f (q ) |>0$.
Thus $f$ has no critical point in $M\setminus B_{R_{0}} (o )$.
Again from Lemmas \ref{lemmatopolog1} and \ref{lemmatopolog}, for
every $0\leq t\leq r(q)$,
\[
\ \int^{t}_{0}\!\!\!\mu_o (s)ds
- \langle \nabla f\circ\gamma_{q},\dot{\gamma}_{q} \rangle
+ \langle \nabla f\circ\gamma_{q},\dot{\gamma_{q}} \rangle
_{|_{s=0}} \!\!\!+g(q)-g(o)
\leq 2 (m-1 )+2c,
\]
so that
\[
\ \frac{d}{ds} f\circ\gamma_{q}\left|_{s=t } \right.
\geq\int^{t}_{0}\mu_o (s )ds
- \{ |\nabla f (o ) |  +2\sup_M|g| +2 (m-1 )+2c \}.
\]
Thus, integrating on $ [2, r (q ) ]$,
\begin{equation*}
\begin{split}
f (q ) & \geq\int^{r (q )}_{2}\int^{t}_{0}
\mu_o(\gamma(s ) )ds  - |f (\gamma_{q} (2 ) ) | \\
& \hskip 1cm  \quad - \bigl \{ |\nabla f (o ) | +2\sup_M|g| +2 (m-1 )2c \bigr \}
\bigl(r (q ) -2 \bigr)
\\ &
\geq \int^{r (q )}_{2}\int^{r (q )}_{0}\mu_o (s )ds -\max_{\partial B_{2} (o )} |f |
\\
& \hskip 1cm  \quad - \bigl\{ |\nabla f (o ) |+2\sup_M |g| +2 (m-1 )+ 2c  \bigr \}
\bigl (r (q )-2\bigr ) \rightarrow +\infty,
\end{split}
\end{equation*}
for $r (q )\rightarrow +\infty$. Therefore $f$ is a smooth exhaustion function whose critical points are confined in a compact set. By standard Morse theory, there exists a compact manifold $N$ with boundary such that $M$ is diffeomorphic to the interior of $N$.

Finally, we  prove (c). Suppose that $\sup_{M} (|\nabla f | +|g|) \leq F<+\infty$.
Then, by (\ref{ric_est}) in Lemma \ref{lemmatopolog},
 for every unit speed geodesic  $\gamma$ issuing from $o$ we have
\[
Ric (\dot{\gamma},\dot{\gamma} )\geq\mu_{0}+\frac{d}{dt}G\circ\gamma,
\]
where $G\circ\gamma
=- \langle \nabla f\circ\gamma,\dot{\gamma} \rangle + g\circ\gamma$
satisfies $ |G\circ\gamma |\leq\sup_M (|\nabla f | +  |g|)$.
Using  Theorem 1.2 in \cite{Ga} we obtain the desired diameter estimate.
\end{proof}

\end{document}